\documentclass[11pt,reqno]{amsart}

\usepackage{xcolor}
\usepackage[T1]{fontenc}
\usepackage{amsmath}							
\usepackage{amssymb}
\usepackage{amsthm}
\usepackage{amscd}
\usepackage{amsfonts}
\usepackage{mathabx}
\usepackage{stmaryrd}
\usepackage[all]{xy}

\usepackage{caption}
\usepackage{subcaption}
\usepackage{euler}

\usepackage{extarrows}

\usepackage[colorlinks, linktocpage, citecolor = purple, linkcolor = blue]{hyperref}
\usepackage{color}

\usepackage{tikz}									
\usetikzlibrary{matrix}
\usetikzlibrary{patterns}
\usetikzlibrary{positioning}
\usetikzlibrary{decorations.pathmorphing}
\usetikzlibrary{cd}

\usepackage{ytableau}
\usepackage{fullpage}
\usepackage[shortlabels]{enumitem}

\newtheorem{theorem}{Theorem}[section]
\newtheorem{lemma}[theorem]{Lemma}
\newtheorem{proposition}[theorem]{Proposition}
 
\newtheorem{conjecture}[theorem]{Conjecture} 

\theoremstyle{definition}
								
\newtheorem{definition}[theorem]{Definition}
\newtheorem{example}[theorem]{Example}

\newtheorem{remark}[theorem]{Remark}
\newtheorem*{remark*}{Remark}

\newcommand{\ZZ}{\mathbb{Z}}

\newcommand{\PP}{\mathbb{P}}

\newcommand{\RR}{\mathbb{R}}

\newcommand{\G}{\mathbb{G}}

\newcommand{\te}{\widetilde{e}}

\newcommand{\thh}{\widetilde{h}}

\newcommand{\tv}{\widetilde{v}}
\newcommand{\tx}{\widetilde{x}}

\newcommand{\tF}{\widetilde{F}}
\newcommand{\tG}{\widetilde{G}}

\newcommand{\tX}{\widetilde{X}}

\newcommand{\tGa}{\widetilde{\Gamma}}

\newcommand{\Ga}{\Gamma}

\newcommand{\De}{\Delta}

\newcommand{\ph}{\varphi}

\DeclareMathOperator{\Ker}{Ker}
\DeclareMathOperator{\Ram}{Ram}

\DeclareMathOperator{\Hom}{Hom}

\DeclareMathOperator{\Vol}{Vol}
\DeclareMathOperator{\Vor}{Vor}

\DeclareMathOperator{\rk}{rk}

\newcommand{\calA}{\mathcal{A}}
\newcommand{\calB}{\mathcal{B}}

\newcommand{\calI}{\mathcal{I}}
\newcommand{\calM}{\mathcal{M}}

\newcommand{\calR}{\mathcal{R}}

\newcommand{\calT}{\mathcal{T}}

\newcommand{\ud}{\mathrm{ud}}	
\newcommand{\st}{\mathrm{st}}	

\DeclareMathOperator{\Aut}{Aut}

\DeclareMathOperator{\Pic}{Pic}

\DeclareMathOperator{\trop}{trop}
\DeclareMathOperator{\val}{val}
\DeclareMathOperator{\dval}{dval}

\DeclareMathOperator{\Id}{Id}
\DeclareMathOperator{\Div}{Div}

\DeclareMathOperator{\Prym}{Prym}

\DeclareMathOperator{\dil}{dil}

\DeclareMathOperator{\Jac}{Jac}

\title[The 2\textsuperscript{nd} moment of the tropical Prym]{The trigonal construction and the second moment of the tropical Prym variety}
 
\author{Dmitry Zakharov}
\address{Department of Mathematics, Central Michigan University, Mount Pleasant, MI 48859, USA}
\email{\href{mailto:dvzakharov@gmail.com}{dvzakharov@gmail.com}}

\begin{document}

\begin{abstract}

    We use the tropical trigonal construction to calculate the second moment of the tropical Prym variety of all double covers $\pi:\tGa\to \Ga$ of tropical curves of genus $g(\Ga)\leq 4$. The answer is expressed in terms of the signed graphic matroid of the double cover and consists a polynomial and piecewise-polynomial term. We relate the latter term, which does not occur in the analogous formula for the tropical Jacobian, to the problem of extending the Prym--Torelli map from the moduli space of admissible double covers to the second Voronoi compactification of the moduli space of principally polarized abelian varieties.


\end{abstract}

\maketitle
\setcounter{tocdepth}{1}
\tableofcontents

\section{Introduction}

Tropical geometry aims to find analogues of algebro-geometric objects in the world of polyhedral geometry, which record the combinatorial behavior of degenerations of the algebraic objects. There are two classes of algebraic objects for which this correspondence is especially well-understood: algebraic curves and abelian varieties. The tropical analogues of algebraic curves are metric graphs, which admit a theory of divisors, linear equivalence, and line bundles that closely follows the algebraic theory (see~\cite{MikhalkinZharkov}). A tropical principally polarized abelian variety (ppav) $X$ is a real torus determined by a free abelian group $\Lambda$ and an inner product on $\Lambda_{\RR}=\Lambda\otimes \RR$, and the theory of tropical ppavs is likewise quite similar to the algebraic theory. For example, the tropical Jacobian $\Jac(\Ga)$ of a metric graph $\Ga$ is a tropical ppav, and there is a natural Abel--Jacobi map $\Ga\to \Jac(\Ga)$ (see~\cite{baker2011metric}).

A tropical ppav is a Riemannian manifold, and a natural first question is to compute its volume. More generally, the \emph{$2n$-th moment} $I_{2n}(X)$ of a tropical ppav $X$ is the integral of $\|\cdot\|^{2n}$ over a canonical symmetric convex polyhedral domain $\Vor(X)\subset \Lambda_{\RR}$, called the \emph{Voronoi polytope}, and zeroth moment is the volume of $X$. The second moment of a tropical ppav has been extensively studied in information theory (see Chapter 21 of~\cite{2013ConwaySloane}), and was recently found to have arithmetic significance~\cite{2022deJongShokrieh}. For a generic tropical ppav $X$, finding $I_{2n}(X)$ (indeed, even finding the vertices of $\Vor(X)$) is a computationally hard problem (see the discussion in Section 1.2 of~\cite{2023deJongShokrieh}). Therefore one is interested in finding explicit formulas for $I_{2n}(X)$ for specific families of tropical ppavs. 

The zeroth moment $I_0(\Jac(\Ga))$, or volume, of the Jacobian of a metric graph $\Ga$ was found in~\cite{2014AnBakerKuperbergShokrieh}, and is equal (up to a multiple) to a polynomial in the edge lengths of $\Ga$ whose monomials correspond to the spanning trees of $\Ga$. A formula for the second moment $I_2(\Jac(\Ga))$ was found in~\cite{2023deJongShokrieh} (see Theorem 8.1). While this formula is computationally efficient for a fixed metric graph, we present our motivation for finding an alternative formula. The Jacobian variety $\Jac(\Ga)$ is closely related to the graphic matroid $\calM(\Ga)$, and in fact can be reconstructed from $\calM(\Ga)$ and the edge lengths (see~\cite{1982Gerritzen},~\cite{CaporasoViviani}, and~\cite{2011BrannettiMeloViviani}). It follows that it should possible to calculate all moments $I_{2n}(\Jac(\Ga))$ in terms of the matroid alone. The formula for $I_0(\Jac(\Ga))$ in~\cite{2014AnBakerKuperbergShokrieh} is a sum over the spanning trees of $\Ga$ and hence is explicitly matroidal, but the formula for $I_2(\Jac(\Ga))$ given in~\cite{2023deJongShokrieh} is not. As an intermediate result, in Proposition~\ref{prop:I2Jac} we use the results of the related paper~\cite{2023RichmanShokriehWu} to derive an explicitly matroidal formula for $I_2(\Jac(\Ga))$, which involves summation over the 2-spanning forests of $\Ga$.

The Prym variety $\Prym(\tX/X)$ of a connected \'etale double cover $p:\tX\to X$ of algebraic curves is a ppav of dimension $g(X)-1$. The analogous construction for double covers of metric graphs was described in~\cite{2018JensenLen} and~\cite{LenUlirsch}. The zeroth moment $I_0(\Prym(\tGa/\Ga))$ of the Prym variety of a double cover $\pi:\tGa\to \Ga$ of metric graphs was computed in~\cite{2022LenZakharov} in the case when $\pi$ is a free double cover, and in~\cite{2024GhoshZakharov} for all double covers. As in the case of Jacobians, $I_0(\Prym(\tGa/\Ga))$ is equal (up to a multiple) to a polynomial in the edge lengths of $\Ga$, with monomials indexed by the bases of a matroid, namely the \emph{signed graphic matroid} $\calM(\tGa/\Ga)$ originally introduced by Zaslavsky~\cite{1982Zaslavsky}. Unlike $\calM(\Ga)$, the bases of the matroid $\calM(\tGa/\Ga)$ come equipped with an index function, which appears as an auxiliary weighting in the corresponding terms of $I_0(\Prym(\tGa/\Ga))$. 

The exact relationship between $\calM(\tGa/\Ga)$ and $\Prym(\tGa/\Ga)$  was elucidated in~\cite{2023RoehrleZakharov}, where we showed that the latter can be reconstructed from the former, just as $\calM(\Ga)$ determines $\Jac(\Ga)$. It follows that there must exist a formula for $I_2(\Prym(\tGa/\Ga))$ in terms of the signed graphic matroid $\calM(\tGa/\Ga)$. The goal of this project is to find such a formula. 

In classical algebraic geometry, a beautiful construction of Recillas~\cite{1974Recillas} associates, to a connected \'etale double cover $\widetilde{X}\to X$ of a trigonal curve $X$, a tetragonal curve $Y$ in such a way that $\Prym(\widetilde{X}/X)=\Jac(Y)$. Under certain genericity assumptions on $Y$, this construction can be inverted. A tropical version of this construction was developed in~\cite{2022RoehrleZakharov}: given a free double cover of metric graphs $\pi:\tGa\to \Ga$, where $\Ga$ is a trigonal graph (in an appropriate sense), we construct a tetragonal metric graph $\Pi$ such that $\Prym(\tGa/\Ga)=\Jac(\Pi)$. By the results of~\cite{cools2018metric}, any metric graph $\Ga$ of genus $g(\Ga)\leq 4$ is trigonal. Therefore, we can compute $I_2(\Prym(\tGa/\Ga))$ for any free double cover with $g(\Ga)\leq 4$ by identifying $\Prym(\tGa/\Ga)$ with a tetragonal Jacobian $\Jac(\Pi)$, using the matroidal formula of Proposition~\ref{prop:I2Jac} to compute $I_2(\Jac(\Pi))$, and then expressing the result in terms of the signed graphic matroid $\calM(\tGa/\Ga)$. By edge contraction, this formula can then be extended to arbitrary double covers. 

The main result of this paper is an implementation of this computation in Sage. In Theorem~\ref{thm:main}, we give an explicit formula for $I_2(\Prym(\tGa/\Ga))$ in terms of the matroid $\calM(\tGa/\Ga)$ for all metric graphs $\Ga$ with $g(\Ga)\leq 4$. The formula consists of two terms: a polynomial term $p(\tGa/\Ga)$ that is the direct analogue of the polynomial formula for $I_2(\Jac(\Ga))$, and a piecewise-polynomial term $q(\tGa/\Ga)$ that has no analogue for Jacobians.

The piecewise-polynomiality of $I_2(\Prym(\tGa/\Ga))$ is directly related to the problem of resolving the indeterminancy of the Prym--Torelli map $\overline{t}_g:\overline{\mathcal{R}}_g\dashrightarrow \overline{\mathcal{A}}^V_{g-1}$ from the moduli space of admissible double covers of stable curves to the second Voronoi compactification of the moduli space of ppavs. In fact, our calculations can be used to describe an explicit sequence of toric blowups that fully resolves the indeterminancy for $g=4$. This blowup, however, is far from minimal, and in future work we plan to use our methods to find a minimal resolution of $\overline{t}_4$.

We emphasize that we do not actually compute the second moment of any tropical abelian variety $X$ that is not a Jacobian, in other words, one whose second moment is not already known by the results of~\cite{2023deJongShokrieh}. However, our results allow us to conjecture a matroidal formula for $I_2(\Prym(\tGa/\Ga))$ for certain families of double covers whose Pryms are not, in general, Jacobians.

\section{Setup}

In this section we recall a number of basic definitions about metric graphs, double covers, graph gonality, and tropical abelian varieties and their moments.

\subsection{Graphs and weighted graphs}

A \emph{graph} $G=(V(G),H(G),r,\iota)$ consists of a set of \emph{vertices} $V(G)$, a set of \emph{half-edges} $H(G)$, a \emph{root map} $r:H(G)\to V(G)$, and a fixed-point-free involution $\iota:H(G)\to H(G)$. An \emph{edge} $e=\{h,h'\}$ of $G$ is an orbit of $\iota$, and the set of edges is denoted $E(G)$. An edge is a loop if its root vertices coincide. We allow graphs with loops and with multiple edges between two vertices, and we consider only finite connected graphs.

The \emph{tangent space} to a vertex $v\in V(G)$ is the set $T_vG=r^{-1}(v)$ of half-edges rooted at $v$, and the \emph{valence} of $v\in V(G)$ is $\val(v)=|T_vG|$. A \emph{weighted graph} $(G,g)$ is a graph $G$ together with a function $g:V(G)\to \ZZ_{\geq 0}$ called the \emph{vertex genus}. The \emph{genus} of a weighted graph is
\[
g(G)=b_1(G)+\sum_{v\in V(G)}g(v)=|E(G)|-|V(G)|+1+\sum_{v\in V(G)}g(v).
\]
The \emph{Euler characteristic} of a vertex $v\in V(G)$ of a weighted graph $G$ is the quantity
\[
\chi(v)=2-2g(v)-\val(v),
\]
and the \emph{Euler characteristic} of the entire graph $G$ is
\[
\chi(G)=\sum_{v\in V(G)}\chi(v)=2-2g(G).
\]
A vertex $v\in V(G)$ is \emph{unstable} if $\chi(v)>0$, such a vertex is an extremal vertex of genus 0. Similarly, $v$ is \emph{semistable} if $\chi(v)=0$, in other words if $\val(v)=2$ and $g(v)=0$. A weighted graph $G$ is called \emph{stable} or \emph{semistable} if respectively $\chi(v)<0$ or $\chi(v)\leq 0$ for all $v\in V(G)$. For any weighted graph $G$ with $g(G)\geq 2$, we define the \emph{stabilization} $G^{\st}$ of $G$ by iteratively removing unstable vertices $v\in V(G)$ (in other words, deleting all extremal trees with no vertices of positive genus), and then removing all semistable vertices by joining the two incident edges into a single edge. 

\subsection{Harmonic morphisms and ramification}

A \emph{morphism} $f:\tG\to G$ of graphs is a pair of maps $f:V(\tG)\to V(G)$ and $f:H(\tG)\to H(G)$ commuting with the root and involution maps, and thus inducing a map $f:E(\tG)\to E(G)$ that preserves vertex adjacency (we do not consider morphisms that contract edges). A \emph{harmonic morphism} of graphs is a morphism $f:\tG\to G$ together with a \emph{local degree} function $d_f:V(\tG)\cup H(\tG)\to \ZZ_{>0}$ satisfying the following conditions: for any edge $\te=\{\thh,\thh'\}\in E(\tG)$ we have $d_f(\thh)=d_f(\thh')$  (a quantity we therefore denote $d_f(\te)$), and 
\[
d_f(\tv)=\sum_{\thh\in f^{-1}(h)\cap T_{\tv}\tG} d_f(\thh)
\]
for any $\tv\in V(\tG)$ and any $h\in T_{f(\tv)}G$. A harmonic morphism $f:\tG\to G$ to a connected target $G$ has a well-defined \emph{global degree} equal to
\begin{equation}
\label{eq:harmonicglobaldegree}    
\deg(f)=\sum_{\tv\in f^{-1}(v)}d_f(\tv)=\sum_{\thh\in f^{-1}(h)}d_f(\thh)=\sum_{\te\in f^{-1}(e)}d_f(\te)
\end{equation}
for any $v\in V(G)$, $h\in H(G)$, or $e\in E(G)$. 

Given a harmonic morphism $f:\tG\to G$ of weighted graphs, the \emph{ramification degree} of $f$ at a vertex $\tv\in V(\tG)$ is 
\begin{equation}
\label{eq:localRH}
\Ram_f(\tv)=d_f(\tv)\chi(f(\tv))-\chi(\tv).
\end{equation}
Adding the local ramification degrees over all $\tv\in V(\tG)$ and using~\eqref{eq:harmonicglobaldegree}, we obtain the \emph{global Riemann--Hurwitz formula}
\begin{equation}
\chi(\tG)=\deg (f)\chi(G)-\Ram(f),
\label{eq:globalRH}
\end{equation}
where 
\[
\Ram(f)=\sum_{\tv\in V(\tG)}\Ram_f(\tv).
\]
is the \emph{global ramification degree} of $f$.

We say that a harmonic morphism $f:\tG\to G$ is \emph{effective} if $\Ram_f(\tv)\geq 0$ for all $\tv\in V(\tG)$ and \emph{unramified} if $\Ram_f(\tv)=0$ for all $\tv\in V(\tG)$. It is elementary to verify that if $f:\tG\to G$ is an unramified harmonic morphism, then $\tG$ is semistable if and only if $G$ is semistable, and stable if and only if $G$ is stable. 

\subsection{Edge contractions} Let $G$ be a graph and let $F\subset E(G)$ be a set of edges. We define the \emph{contraction} $G_F$ of $G$ along $F$ as follows. Let $G[F]=G_1\sqcup\cdots\sqcup G_k$ be the connected components of the subgraph $G[F]\subseteq G$ spanned by $F$ (consisting of the edges in $F$ and their root vertices). We form $G_F$ by contracting each $G_i$ to a separate vertex $v_i$, so that
\[
V(G_F)=(V(G)\backslash V(G[F]))\cup\{v_1,\ldots,v_k\},\quad E(G_F)=E(G)\backslash F.
\]
If $G$ is a weighted graph with vertex genus function $g:V(G)\to \ZZ_{\geq 0}$, we define the vertex weight $g_F:V(G_F)\to \ZZ_{\geq 0}$ as follows:
\[
g_F(v)=\begin{cases} g(v), & v\in V(G)\backslash V(G[F]),\\
g(G_i), & v=v_i.
\end{cases}
\]
It is elementary to verify that $g(G_F)=g(G)$.

We can similarly contract a harmonic morphism $f:\tG\to G$ along a set of edges $F\subset E(G)$ of the target. Let $G[F]=G_1\sqcup\cdots\sqcup G_k$ be the connected components, let $\tF=f^{-1}(F)$, and let $\tG[\tF]=\coprod_{i,j} \tG_{ij}$ be the connected components of the subgraph of $\tG$ spanned by $\tF$, labeled in such a way that $f(\tG_{ij})=G_i$. We define the \emph{contraction $f_F:\tG_{\tF}\to G_F$ of $f$ along $F$} by setting $f(\tv_{ij})=v_i$, where $\tv_{ij}\in V(\tG_{\tF})$ and $v_i\in V(G_F)$ are the contracted vertices corresponding to $\tG_{ij}$ and $G_i$, respectively. It is elementary to verify that $f_F$ is harmonic and of the same global degree as $f$, provided that we set its local degrees to be
\[
d_{f_F}(\tv)=\begin{cases} d_f(\tv), & \tv\in V(\tG)\backslash V(\tG[\tF]), \\
\deg f|_{\tG_{ij}}, & \tv=\tv_{ij}.    
\end{cases}
\]
Furthermore, if $f$ is effective or unramified, then so is $f_F$.

\subsection{Metric graphs and tropical curves}

Let $G$ be a graph and let $\ell:E(G)\to \RR_{>0}$ be an assignment of \emph{lengths} to the edges of $G$. The pair $(G,\ell)$ defines a \emph{metric graph} $\Ga$, which is the topological space obtained by gluing intervals $[0,\ell(e)]$ for all $e\in E(G)$ according to how they are attached in $G$. We will often use $e$ to denote both an edge of a metric graph and its length. The pair $(G,\ell)$ is called a \emph{model} for $\Ga$, and different models for $\Ga$ may be obtained by either subdividing edges or removing valence 2 vertices. Any metric graph other than a circle has a unique \emph{minimal model}, having no vertices of valence 2. We will usually not distinguish a metric graph $\Ga$ from its models and use $V(\Ga)$ and $E(\Ga)$ to denote the set of vertices or edges of a suitably chosen model of $\Ga$.

A \emph{tropical curve} is a weighted metric graph, in other words a metric graph $\Ga$ together with a weight function $g:\Ga \to \ZZ_{\geq 0}$ that takes only finitely many nonzero values. When choosing a model for a tropical curve $\Ga$, we always assume that each point $x\in \Ga$ with $g(x)>0$ corresponds to a vertex, so that the model is a weighted graph. We define the Euler characteristic of a point $x\in \Ga$ on a tropical curve $\Ga$, and the genus and Euler characteristic of all of $\Ga$, using the underlying model. A tropical curve is called \emph{stable} if it has no extremal vertices of genus zero. 

A \emph{harmonic morphism} $f:\tGa\to \Ga$ of metric graphs or tropical curves is determined by a harmonic morphism of the underlying models, having the property that
\[
\ell(f(\te))=d_f(\te)\ell(\te)
\]
for all $\te\in E(\tGa)$. The map $f$ is linear on each edge $\tGa$ and dilates it by the positive integer factor $d_f(\te)$. The global degree of $f$ is its global degree as a harmonic morphism of graphs, and we similarly refer to effective and unramified harmonic morphisms of tropical curves.

We may contract edges of tropical curves by contracting the edges on the underlying models, and similarly we may contract harmonic morphisms of tropical curves. Intuitively, this corresponds to setting the lengths of the corresponding edges to zero, and increasing the vertex genus by one whenever a loop is contracted.

\subsection{Double covers} In this paper, we are primarily concerned with two types of harmonic morphisms: double covers and $n$-gonal structures, which are described in Section~\ref{subsec:gonality}.

A \emph{double cover} $\pi:\tGa\to \Ga$ of tropical curves is an unramified harmonic morphism of global degree two. A point $x\in \Ga$ is called \emph{dilated} if it has a unique preimage $\pi^{-1}(x)=\{\tx\}$ with $d_{\pi}(\tx)=2$ and \emph{free} if $\pi^{-1}(x)=\{\tx^+,\tx^-\}$ with $d_{\pi}(\tx^{\pm})=1$. The set of dilated points forms the  \emph{dilation subcurve} $\Ga_{\dil}\subset \Ga$. A double cover $\pi:\tGa\to \Ga$ is \emph{free} if $\Ga_{\dil}=\emptyset$, in other words, if $\pi$ is a covering isometry. We say that $\pi$ is \emph{dilated} if $\Ga_{\dil}\neq\emptyset$ and \emph{edge-free} of $\Ga_{\dil}$ consists of finitely many points. The \emph{dilation index} of a double cover is
\[
d(\tGa/\Ga)=\begin{cases} \mbox{number of connected components of }\Ga_{\dil},& \pi\mbox{ is dilated,}\\
1\mbox{ (not 0)}, & \pi\mbox{ is free}.
\end{cases}
\]
The \emph{torus rank} of a double cover $\pi:\tGa\to \Ga$ is
\[
t(\tGa/\Ga)=b_1(\tGa)-b_1(\Ga).
\]

\begin{example} The \emph{$FS_n$-double cover} $\pi:\tGa\to \Ga$ is shown on Figure~\ref{fig:FSn}. The two vertices of $\Ga$ are dilated, and no edges are dilated. The dilation index and torus rank are equal to $d(\tGa/\Ga)=2$ and $t(\tGa/\Ga)=n$, respectively. 
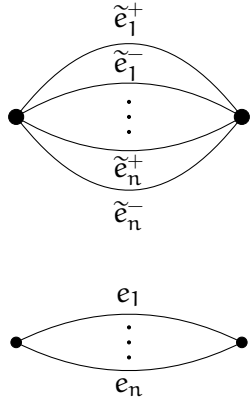
\begin{figure}
    \centering
    \begin{tikzpicture}

    \draw[fill] (0,0) circle(0.07);
    \draw[fill] (3,0) circle(0.07);
    \draw[thin] (0,0) .. controls (1,0.5) and (2,0.5) .. (3,0);
    \draw[thin] (0,0) .. controls (1,-0.5) and (2,-0.5) .. (3,0);
    \draw[fill] (1.5,0) circle (0.02);
    \draw[fill] (1.5,0.2) circle (0.02);
    \draw[fill] (1.5,-0.2) circle (0.02);
    \node at (1.5,0.6) {$e_1$};
    \node at (1.5,-0.6) {$e_n$};

\begin{scope}[xshift=0cm,yshift=3cm]

    \draw[fill] (0,0) circle(0.1);
    \draw[fill] (3,0) circle(0.1);
    \draw[thin] (0,0) .. controls (1,1.3) and (2,1.3) .. (3,0);
    \draw[thin] (0,0) .. controls (1,-1.3) and (2,-1.3) .. (3,0);
    \draw[thin] (0,0) .. controls (1,0.6) and (2,0.6) .. (3,0);
    \draw[thin] (0,0) .. controls (1,-0.6) and (2,-0.6) .. (3,0);
    \draw[fill] (1.5,0) circle (0.02);
    \draw[fill] (1.5,0.2) circle (0.02);
    \draw[fill] (1.5,-0.2) circle (0.02);
    \node at (1.5,1.3) {$\te^+_1$};
    \node at (1.5,-0.7) {$\te^+_n$};
    \node at (1.5,0.7) {$\te^-_1$};
    \node at (1.5,-1.3) {$\te^-_n$};
\end{scope}
    \end{tikzpicture}
    \caption{The $FS_n$-double cover}
    \label{fig:FSn}
\end{figure}
    
\end{example}

We make a number of elementary numerical observations about double covers.

\begin{lemma} Let $\pi:\tGa\to \Ga$ be a double cover and let $\Ga_1,\ldots,\Ga_d$ be the connected components of the dilation subcurve $\Ga_{\dil}$.

\begin{enumerate}
    \item $g(\tGa)=2g(\Ga)-1$. 
    \item Each $\Ga_i$ is semistable (see Lemma 5.4 in~\cite{2018JensenLen}).
    \item The torus rank satisfies
\[
t(\tGa/\Ga)\leq g(\Ga)-1,
\]
    with equality holding if and only if the following conditions hold:
    \begin{enumerate}
        \item $g(\Ga_i)=1$ for all $i$, in other words each $\Ga_i$ is either an isolated point of genus one or a simple cycle having no points of nonzero genus.
        \item $g(x)=0$ for all $x\notin \Ga_{\dil}$.
    \end{enumerate}
\end{enumerate}
\label{lem:torusrank}
\end{lemma}

\begin{proof} The first equality follows from the global Riemann--Hurwitz formula~\eqref{eq:globalRH}. To prove the second part, denote by $\dval(x)$ the valence of a dilated point $x\in \Ga_{\dil}$ in the subgraph $\Ga_{\dil}$. Since $\pi$ is unramified, the local Riemann--Hurwitz condition~\eqref{eq:localRH} at the preimage point $\tx=\pi^{-1}(x)$ reads
\[
g(\tx)=2g(x)-1+\frac{1}{2}\dval(x).
\]
Hence $\dval(x)$ is even, and furthermore $g(x)\geq 1$ if $\dval(x)=0$. Hence each $\Ga_i$ is semistable, and in particular $g(\Ga_i)\geq 1$. 

By counting the preimage vertices and edges, it is elementary to verify that
\[
t(\tGa/\Ga)=b_1(\Ga)-1-\sum_{i=1}^d[b_1(\Ga_i)-1]=g(\Ga)-1-\sum_{i=1}^d[g(\Ga_i)-1]-\sum_{x\notin \Ga_{\dil}}g(x).
\]
Since $g(\Ga_i)\geq 1$ for all $i$, this proves the third part of the lemma.
    
\end{proof}

A free double cover $\pi:\tGa\to \Ga$ on a given tropical curve $\Ga$ is equivalent to the structure of a \emph{signed graph} on $\Ga$. Choose an oriented model for $\Ga$, and for each $v\in V(\Ga)$ label the preimages by $\pi^{-1}(v)=\{\tv^+,\tv^-\}$ arbitrarily. For each $e\in E(\Ga)$, we label the preimages $\pi^{-1}(e)=\{\te^+,\te^-\}$ in such a way that $s(\te^{\pm})=\widetilde{s(e)}^{\pm}$, so that the source map preserves the signs on the preimages. The edges of $\Ga$ are then equipped with signs $\sigma:E(\Ga)\to \{+1,-1\}$ according to whether or not the target map preserves or flips the signs of the preimages:
\[
\sigma(e)=\begin{cases}
    +1, & t(\te^{\pm})=\widetilde{t(e)}^{\pm},\\
    -1, & t(\te^{\pm})=\widetilde{t(e)}^{\mp}.
\end{cases}
\]
Exchanging the labels of the preimages $\{\tv^+,\tv^-\}$ of a single vertex $v\in V(\Ga)$ flips the signs of all non-loop edges at $v$, so the simplicial cohomology class $[\sigma]\in H^1(\Ga,\ZZ/2\ZZ)$ is well-defined. Running the construction in reverse, we obtain a free double cover $\pi:\tGa\to \Ga$ from an element of $H^1(\Ga,\ZZ/2\ZZ)$, with the zero element corresponding to the split double cover (for any $\tx\in \tGa$ we have $g(\tx)=g(\pi(\tx))$ by the local Riemann--Hurwitz condition).

As we noted above, the contraction of an unramified morphism is unramified, in particular, the contraction of a double cover $\pi:\tGa\to \Ga$ along a set of edges $F\subset E(\Ga)$ is a double cover $\pi_F:\tGa_{\tF}\to \Ga_F$. We note that $\pi_F$ may be dilated even if $\pi$ is free, this happens if and only if $F$ contains an odd cycle.

\subsection{Moduli of tropical curves and double covers} Let $G$ be a stable weighted graph of genus $g\geq 2$, so that if $g(v)=0$ for $v\in V(G)$ then $\val(v)\geq 3$. It is elementary to verify that $|E(G)|\leq 3g-3$, with equality holding if and only if $G$ is trivalent and has trivial vertex genera. A tropical curve with underlying model $G$ is obtained by choosing positive lengths $\ell(e)$ for all $e\in E(G)$, so the positive orthant $M_G=\RR^{E(G)}_{>0}$ (more accurately, the quotient $\RR^{E(G)}_{>0}/\Aut(G)$) parametrizes stable tropical curves with underlying model $G$. 

Now let $F\subset E(G)$ be a set of edges and let $G_F$ be the contraction of $G$ along $F$. We attach $M_{G_F}$ to $M_G$ as the face of the orthant where the lengths of the edges in $F$ are equal to zero. Similarly, if $f:G\to H$ is an isomorphism stable weighted graphs, then we identify $M_G$ and $M_H$ via the corresponding bijection $E(G)\to E(H)$. The union of the cones $M_G$ over all stable weighted metric graphs of genus $g$, identified as above, is the \emph{moduli space of tropical curves} of genus $g$:
\[
\calM_g^{\trop}=\left.\bigsqcup_{G}M_G\right/\sim
\]
Since each $G$ has at most $3g-3$ edges, the dimension of $\calM_g^{\trop}$ is equal to $3g-3$, and each cone is contained in the boundary of a cone of maximal dimension.

We similarly define the \emph{moduli space of double covers of tropical curves} of genus $g$
\[
\calR_g^{\trop}=\left.\bigsqcup_{\tG/G}R_{\tG/G}\right/\sim
\]
Here the union is taken over all connected double covers $p:\tG\to G$ of stable weighted graphs of genus $g(G)=g$. Each $R_{\tG/G}=\RR_{>0}^{E(G)}$ is a cone parametrizing double covers of tropical curves $\pi:\tGa\to \Ga$ with underlying model $p$ (the edge lengths of $\Ga$ determine those of $\tGa$). Given a double cover $p:\tG\to G$ and a set of edges $F\subset E(G)$, the contraction $p_F:\tG_{\tF}\to G_F$ along the edges in $F$ is also unramified and hence a double cover, and we attach $R_{\tG_{\tF}/G_F}$ to $R_{\tG/G}$ accordingly. Similarly, we glue the cones along isomorphisms of double covers. The maximal-dimensional cones in $\calR^{\trop}_g$ have dimension $3g-3$ and correspond to free double covers $p:\tG\to G$, where $G$ is trivalent and has trivial vertex genera (by Lemma~\ref{lem:torusrank}, such a $G$ cannot be the target of a dilated double cover). We note that by contracting a free double cover we may obtain edge-free double covers but not double covers with dilated edges, hence the closure of the union of the maximal-dimensional cones of $\calR_g^{\trop}$ is the set of edge-free double covers and not all of $\calR_g^{\trop}$.

\subsection{Graph gonality}\label{subsec:gonality} We recall that an algebraic curve $X$ is called \emph{$n$-gonal} if it carries a $g^1_n$, or equivalently if it admits a degree $n$ map to $\mathbb{P}^1$. The tropical analogues of these two definitions are not equivalent, and in fact the second is stronger than the first. One may say that a metric graph $\Ga$ is $n$-gonal if it carries a $g^1_n$ in the sense of Baker and Norine~\cite{baker2007riemann}. Baker's specialization lemma states that the tropicalization of an $n$-gonal algebraic curve is an $n$-gonal metric graph; however, since that lemma is an inequality, gonality may increase under tropicalization. One consequence is that Baker--Norine gonality loci in $\calM_g^{\trop}$ are poorly behaved and generally have incorrect dimensions. 

We use the alternative definition of gonality in terms of the existence of a harmonic map to a metric tree. This definition also allows us to reformulate the classical trigonal construction in the tropical setting, which is the main technical tool of this paper.

\begin{definition} Let $\Ga$ be a tropical curve. A \emph{tropical modification} of $\Ga$ is a tropical curve obtained by attaching finitely many metric trees to $\Ga$.
\end{definition}

The points of the new metric trees are assumed to have genus zero, so tropical modification does not change the genus. We now tropicalize the second definition of gonality by replacing $\mathbb{P}^1$ with a tropical curve of genus zero, that is to say, a metric tree.

\begin{definition} Let $\Ga$ be a tropical curve. An \emph{$n$-gonal structure} on $\Ga$ is an effective harmonic morphism $\ph:\Ga'\to \De$ of degree $n$, where $\Ga'$ is a tropical modification of $\Ga$ and $\De$ is a metric tree of genus zero. A tropical curve $\Ga$ is called \emph{$n$-gonal} if it admits an $n$-gonal structure.

\end{definition}

This definition of gonality was advanced in the paper~\cite{cools2018metric}, where it was shown that it satisfies the same properties, in terms of dimension in moduli, as the gonality of an algebraic curve:

\begin{theorem}[Theorem 1 in~\cite{cools2018metric}] Let $\Ga$ be a tropical curve of genus $g$ and let $n\geq \lceil(g+2)/2\rceil$. Then there exists an $n$-gonal structure $\ph:\Ga'\to \De$ on $\Ga$. \label{thm:CD}

\end{theorem}

We note that Theorem 1 in~\cite{cools2018metric} concerns only metric graphs, but the result also trivially applies to tropical curves with positive vertex genera, since increasing $g$ without changing the underlying metric graph only strengthens the condition on $n$.

One of the key observations of~\cite{cools2018metric} is that the number of edges of the target tree $\De$, and hence the dimension of the space of $n$-gonal curves in $\calM_g^{\trop}$, may be bounded in terms of $n$ and $g$. The following statement appears as Corollary 14 in~\cite{cools2018metric}; however, that paper is written in somewhat different language than ours (and for metric graphs only), so we reprove it for the reader's convenience.

\begin{proposition} Let $\Ga$ be a tropical curve of genus $g$ and let $n\geq \lceil(g+2)/2\rceil$. Then there exists an $n$-gonal structure $\ph:\Ga'\to \De$ such that the target tree $\De$ has at most $2g+2n-5$ edges.
\label{prop:2g+2n-5}
\end{proposition}

\begin{proof} Let $\ph:\Ga'\to \De$ be any $n$-gonal structure on $\Ga$, which exists by Theorem~\ref{thm:CD}. By definition,
\[
\Ram_{\varphi}(\tv)=d_\varphi(\tv)\chi(v)-\chi(\tv)=d_\varphi(\tv)(2-\val(v))-2+2g(\tv)+\val(\tv)\geq 0
\]
for all $\tv\in V(\Ga')$, where $v=\varphi(\tv)$.

Now let $v\in V(\De)$ be an extremal vertex at the end of an edge $e\in E(\De)$, and let $\tv\in \ph^{-1}(v)$. The above inequality implies that if $\Ram_{\ph}(\tv)=0$ then $g(\tv)=0$, $d_\ph(\tv)=1$, and $\val(\tv)=1$, and if $\Ram_{\ph}(\tv)=1$ then $g(\tv)=0$, $d_\ph(\tv)=2$, and $\val(\tv)=1$. In both cases, $\tv$ is an extremal genus zero vertex at the end of an edge of $\Ga'$ lying above $e$, along which $\ph$ has degree $1$ or $2$. Therefore, if $\Ram_{\ph}(\tv)\leq 1$ for all $\tv\in \ph^{-1}(v)$, then we can modify $\ph:\Ga'\to \De$ by removing the edge $e$ and its preimages. Similarly, let $v\in V(\De)$ be a vertex of valence 2. If $\Ram_{\ph}(\tv)=0$ for $\tv\in \ph^{-1}(v)$, then $g(\tv)=0$ and $\val(\tv)=2$, and if this holds for all $\tv\in \ph^{-1}(v)$, then $v$ and all vertices over it may be removed from the models.

Applying these modifications, we may assume that $\ph:\Ga'\to \De$ has the following properties:

\begin{enumerate}
    \item The total ramification degree over each $v\in V(\De)$ with $\val(v)=1$ is at least 2.

    \item The total ramification degree over each $v\in V(\De)$ with $\val(v)=2$ is at least 1.
\end{enumerate}
Now let $a_i$ be the number of vertices of $\De$ of valence $i$. By the global Riemann--Hurwitz formula~\eqref{eq:globalRH}, the ramification degree of $\ph$ is equal to
\[
\Ram(\ph)=n\chi(\De)-\chi(\Ga)=2g+2n-2\geq 2a_1+a_2. 
\]
On the other hand, we know that
\[
a_1+2a_2+3a_3+\cdots=2|E(\De)|,\quad a_1+a_2+a_3+\cdots=|V(\De)|=|E(\De)|+1,
\]
which implies that
\[
|E(\De)|\leq 2a_1+a_2-3\leq 2g+2n-5.
\]

\end{proof}

\subsection{Tropical ppavs and tropical moments} A \emph{tropical ppav} $X=(\Lambda,\Lambda',[\cdot,\cdot],\zeta)$ of dimension $d$ is determined by a pair of free abelian groups $\Lambda$ and $\Lambda'$ of rank $d$, a nondegenerate pairing $[\cdot,\cdot]:\Lambda\times \Lambda'\to \RR$, and an isomorphism $\zeta:\Lambda'\to \Lambda$ that induces a symmetric positive definite bilinear form $(\cdot,\cdot)=[\zeta(\cdot),\cdot]$. The map $\lambda'\mapsto [\cdot,\lambda']$ defines an inclusion $\Lambda'\subset \Lambda^*_{\RR}=\Hom(\Lambda,\RR)$, and the ppav itself is the $d$-dimensional torus $\Hom(\Lambda,\RR)/\Lambda'$. The inner product $(\cdot,\cdot)$ induces a norm $\|\cdot\|$ on $\Lambda^*_{\RR}$, and the \emph{Voronoi polytope} of $X$ is the convex symmetric polyhedral domain
\[
\Vor(X)=\{x\in \Lambda^*_{\RR}:\|x\|\leq \|x-y\|\mbox{ for all }y\in \Lambda'\}.
\]
The \emph{$2n$-th moment} $I_{2n}(X)$ is the integral
\[
I_{2n}(X)=\int_{\Vor(X)}\|x\|^{2n}dx.
\]
Since the Voronoi polytope is a fundamental domain for the lattice $\Lambda'$, the zeroth moment computes the volume of $X$ as a Riemannian manifold:
\[
I_0(X)=\Vol(X),
\]
while the second moment
\[
I_2(X)=\int_{\Vor(X)}\|x\|^2dx
\]
has arithmetic significance (see~\cite{2022deJongShokrieh}).

In this paper, we consider two examples of tropical ppavs. Let $\Ga$ be a metric graph and let $H_1(\Ga,\ZZ)$ be the simplicial homology group of $\Ga$. We interpret $H_1(\Ga,\ZZ)$ as the space $\Omega^1(\Ga)$ of \emph{harmonic 1-forms} on $\Ga$, by viewing a 1-cycle $\sum_{e\in E(\Ga)} a_ee\in H_1(\Ga,\ZZ)$ as a 1-form $\sum_{e\in E(\Ga)} a_ede\in \Omega^1(\Ga)$. The lattices $\Lambda=\Omega^1(\Ga)$, $\Lambda'=H_1(\Ga,\ZZ)$, the \emph{integration pairing}
\begin{equation}
\Omega^1(\Ga)\times H_1(\Ga,\ZZ)\to \RR,\quad
\left[\sum a_ede,\sum b_ee\right]=\sum a_eb_e\ell(e),
\label{eq:integrationpairing}
\end{equation}
and the trivial isomorphism $H_1(\Ga,\ZZ)=\Omega^1(\Ga)$ define the \emph{tropical Jacobian} $\Jac(\Ga)$, a tppav of dimension $g(\Ga)=\dim H_1(\Ga,\ZZ)$. Tropical Abel--Jacobi theory identifies $\Jac(\Ga)$ with $\Pic_0(\Ga)$, the set of equivalence classes of degree zero divisors on $\Ga$.

The \emph{tropical Prym variety} $\Prym(\tGa/\Ga)$ of a double cover $\pi:\tGa\to \Ga$ of metric graphs was defined in~\cite{2018JensenLen} in analogy with the algebraic setting, as the connected component of the identity of the kernel of the norm map
\[
\pi_*:\Jac(\tGa)\to \Jac(\Ga)
\]
on divisors. However, this definition does not behave well under edge contractions of $\Ga$. Specifically, the Prym variety $\Prym(\tGa_F/\Ga_F)$ of a double cover $\pi_F:\tGa_F\to \Ga_F$ obtained from a double cover $\pi:\tGa\to \Ga$ by contracting a set of edges $F\subset E(\Ga)$ is not equal to the limit of $\Prym(\tGa/\Ga)$ if $d(\tGa_F/\Ga_F)>d(\tGa/\Ga)$. Hence $\Prym(\tGa/\Ga)$, which we henceforth call the \emph{divisorial Prym variety}, does not behave well in moduli and is not a convenient object for our purposes.

This problem was rectified in the paper~\cite{2022RoehrleZakharov}, where we constructed an alternative object, the \emph{continuous Prym variety} $\Prym_c(\tGa/\Ga)$ of a double cover $\pi:\tGa\to \Ga$. The involution $\iota:\tGa\to\tGa$ associated to $\pi$ defines pullback and pushforward maps
\[
\iota^*:\Omega^1(\tGa)\to \Omega^1(\tGa),\quad
\iota_*:H_1(\tGa,\ZZ)\to H_1(\tGa,\ZZ),
\]
and we consider the lattices
\[
\Lambda=\Omega^1(\tGa)/\Ker(\Id-\iota^*),\quad \Lambda'=\operatorname{Im}\iota_*.
\]
The restriction of the isomorphism $H_1(\tGa,\ZZ)=\Omega^1(\tGa)$ to $\Lambda'$ defines a map $\Lambda'\to \Lambda$ than can be divided by two to obtain a natural isomorphism $\zeta:\Lambda'\to \Lambda$. The continuous Prym variety is the dimension $t(\tGa/\Ga)$ tropical ppav associated to the lattices $\Lambda$ and $\Lambda'$, the restriction of the integration pairing, and the isomorphism $\zeta$. It is elementary to verify that the lattices $\Lambda$ and $\Lambda'$ behave well under contraction, and hence we have the following result:

\begin{proposition} Let $\pi:\tGa\to \Ga$ be a double cover, let $F\subset E(\Ga)$ be a set of edges, and let $\pi_F:\tGa_{\tF}\to \Ga_F$ be the double cover obtained by contracting the edges in $F$. The Prym variety $\Prym(\tGa_{\tF}/\Ga_F)$ is the limit of the Prym variety $\Prym(\tGa/\Ga)$ as $\ell(e)\to 0$ for all $e\in F$. In particular, for all $n\geq 0$ we have
\[
I_{2n}(\Prym_c(\tGa_{\tF}/\Ga_F))=\lim_{\ell(e)\to 0,\,e\in F} I_{2n}(\Prym_c(\tGa/\Ga)).
\]
\label{prop:Prymlimit}
\end{proposition}

We also observe that neither $\Prym_c(\tGa/\Ga)$ nor $\Prym(\tGa/\Ga)$ change when we contract a dilated edge of $\Ga$, since the lattices $\Lambda$ and $\Lambda'$ do not change. In other words, we have the following result: 

\begin{proposition} Let $\pi:\tGa\to \Ga$ be a double cover and let $F=E(\Ga_{\dil})$ be the set of dilated edges of $\Ga$. Then $\pi_{\tF}:\tGa_{\tF}\to \Ga_F$ is an edge-free double cover and 
\[
\Prym(\tGa_{\tF}/\Ga_F)\simeq\Prym(\tGa/\Ga).
\]
\label{prop:dilatededgesPrym}
In particular, for all $n\geq 0$ we have
\[
I_{2n}(\Prym_c(\tGa_{\tF}/\Ga_F))=I_{2n}(\Prym_c(\tGa/\Ga)).
\]
\end{proposition}

These two results imply that it is sufficient to compute the second moment $I_2(\Prym(\tGa/\Ga))$ for all free double covers, since the second moment for any edge-free double cover, and hence any double cover, can then be obtained as a limit.

The continuous Prym variety comes equipped with a natural map $\Prym_c(\tGa/\Ga)\to \Prym(\tGa/\Ga)$ to the divisorial Prym variety of degree $2^{d(\tGa/\Ga)-1}$. In particular, $\Prym_c(\tGa/\Ga)=\Prym(\tGa/\Ga)$ when $d(\tGa/\Ga)=1$, in other words when $\pi$ is free or has connected dilation subgraph. Our main tool for computing the second moment of the tropical Prym variety, the trigonal construction, is only defined for free double covers. Hence, in practice we do not need to distinguish $\Prym_c(\tGa/\Ga)$ and $\Prym(\tGa/\Ga)$.

\begin{remark} The Jacobian $\Jac(\Ga)$ of a metric graph $\Ga$ is a tropical ppav of dimension $g$, so it is natural to assume that its volume should be a degree $g$ function in the edge lengths of $\Ga$. Note, however, that the integration pairing~\eqref{eq:integrationpairing}, and hence the inner product that it defines, is linear, rather than quadratic, in the edge lengths. Hence units of length on $\Ga$ become units of area on $\Jac(\Ga)$, and the volume of $\Jac(\Ga)$ is a function of degree $g/2$ in the edge lengths of $\Ga$. More generally, $I_{2n}(\Jac(\Ga))$ is a function of degree $g/2+n$, while $I_{2n}(\Prym(\tGa/\Ga))$ is a function of degree $t(\tGa/\Ga)/2+n$ in the edge lengths of $\Ga$. 

\label{rem:dimension}
\end{remark}

\subsection{The trigonal construction} Finally, we recall the tropical trigonal construction, which is our principal tool for computing the second moment of the tropical Prym variety. Given a connected \'etale double cover $p:\tX\to X$ of an algebraic curve $X$ with a fixed trigonal map $X\to \PP^1$, Recillas~\cite{1974Recillas} constructs a tetragonal curve $Y\to \PP^1$ such that the Jacobian $\Jac(Y)$ is isomorphic to the Prym variety $\Prym(\tX/X)$ (in particular, $g(Y)=g(X)-1$). In~\cite{2022RoehrleZakharov}, we defined a tropical analogue of Recillas's construction.

Let $\pi:\tGa\to \Ga$ be a free double cover of a trigonal tropical curve $\ph:\Ga\to \De$ (if it is necessary to attach trees to $\Ga$ to define $\ph$, we attach two copies of each tree to $\tGa$ and extend $\pi$ trivially). For $x\in \De$, define the pullback
\[
\ph^*(x)=\sum_{y\in \ph^{-1}(x)}d_{\ph}(y)y\in \Div(\Ga),
\]
the free abelian group on the points of $\Ga$. Similarly, we define the pushforward of a divisor on $\tGa$:
\[
\pi_*\sum_{\widetilde{y}\in \tGa} a_{\widetilde{y}}\widetilde{y}=\sum_{\widetilde{y}\in \tGa} a_{\widetilde{y}}\pi(\widetilde{y}).
\]
We now consider the set of sections of the fibers of $\pi:\tGa\to \Ga$ over the various points of $\Delta$:
\[
\widetilde{\Pi}=\left\{D\in \Div(\tGa):D\geq 0\mbox{ and }\pi_*(D)=\ph^*(x)\mbox{ for some (necessarily unique) }x\in \Delta\right\}.
\]

\begin{theorem}[Theorem 5.1 in~\cite{2022RoehrleZakharov}] Let $\pi:\tGa\to \Ga$ be a connected free double cover of metric graphs and let $\ph:\Ga\to \De$ be a harmonic morphism of degree 3, where $\De$ is a metric tree. \label{thm:trigonal} 

\begin{enumerate}
    \item There is a natural structure of a metric graph on $\widetilde{\Pi}$ such that the map $\widetilde{\psi}:\widetilde{\Pi}\to \De$, which sends $D\in \widetilde{\Pi}$ to the unique $x\in \Delta$ such that $\pi_*(D)=\ph^*(x)$, is a harmonic morphism of degree 8.

    \item The metric graph $\widetilde{\Pi}$ consists of two isomorphic connected components $\Pi$ that are exchanged by the involution that is induced by the fixed-point-free involution $\iota:\tGa\to \tGa$ corresponding to the double cover, and the restriction $\psi:\Pi\to \De$ of $\widetilde{\psi}$ to each component is a harmonic morphism of degree 4. 
    \item The metric graph $\Pi$ has genus $g(\Pi)=g(\Ga)-1$, and the Prym variety of the free double cover $\pi:\tGa\to \Ga$ and the Jacobian variety of $\Pi$ are isomorphic as tropical principally polarized abelian varieties:
\[
\Prym_c(\tGa/\Ga)\simeq \Jac(\Pi).
\]

\end{enumerate}

\end{theorem}

We note again that $\Prym(\tGa/\Ga)=\Prym_c(\tGa/\Ga)$ for free double covers. We also note that the paper~\cite{2022RoehrleZakharov} did not consider vertex genera, in particular $\ph$ is not required to be effective. If $\Ga$ has vertex genera and $\ph$ is effective, then there is a natural way to assign vertex genera to $\Pi$ in such a way that $g(\Pi)=g(\Ga)-1$ and that $\psi$ is effective as well. We do not require this and leave the details to the avid reader.

\section{Tropical moments and matroids}

In this section, we recall the known formulas for the moments of tropical abelian varieties. The zeroth moments (the volumes) of the Jacobian and the Prym were computed in~\cite{2014AnBakerKuperbergShokrieh} and~\cite{2022LenZakharov}, respectively. We restate these formulas in terms of certain matroids: the graphic matroid $\calM(\Ga)$ determines the volume of the Jacobian $\Jac(\Ga)$, while the signed graphic matroid $\calM(\tGa/\Ga)$ determines the volume of the Prym variety $\Prym(\tGa/\Ga)$ of a double cover $\tGa\to \Ga$. The second moment of the Jacobian was found in~\cite{2023deJongShokrieh}. However, the formula given in~\cite{2023deJongShokrieh} is not matroidal, and in Proposition~\ref{prop:I2Jac} we find an equivalent matroidal formula, which forms the basis of our matroidal formula for the second moment of the Prym variety.

\subsection{Moments of the tropical Jacobian in terms of the graphic matroid} Let $\Ga$ be a metric graph together with a choice of underlying model that we denote by $\Ga$ by abuse of notation. Let $m=|E(\Ga)|$ and $n=|V(\Ga)|$, so that the genus is $g=g(\Ga)=m-n+1$. We recall the definition of the graphic matroid $\calM(\Ga)$ of $\Ga$. The ground set is the set of edges $E(\Ga)$, and an edge set $F\subset E(\Ga)$ is independent if the subgraph $\Ga[F]$ spanned by $F$ contains no cycles. The bases $\calB(\calM(X))$ are the edge sets of the spanning trees, so that the rank of $\calM(\Ga)$ is the number of edges in a spanning tree and is equal to
\[
\rk(\calM(X))=m-g=n-1.
\]
Finally, the circuits of $\calM(\Ga)$ are the simple closed cycles. 

Since $\calM(\Ga)$ contains information about the cycles of $\Ga$, it is reasonable to expect that the Jacobian $\Jac(\Ga)$ can be reconstructed solely from the graphic matroid $\calM(\Ga)$ and the edge length assignment $\ell:E(\Ga)\to \RR_{>0}$ on its ground set. This is indeed the case, and the relationship between $\calM(\Ga)$ and $\Jac(\Ga)$ was explored in detail in~\cite{2010CaporasoViviani} and \cite{2011BrannettiMeloViviani} (see also~\cite{1982Gerritzen}). In particular, we should be able to calculate all moments of $\Jac(\Ga)$ directly from the matroid $\calM(\Ga)$.

The zeroth moment, or volume, of the tropical Jacobian was computed in~\cite{2014AnBakerKuperbergShokrieh}, and we recall this result. Let $T\in \calB(\calM(\Ga))$ be the edge set of a spanning tree of $\Ga$. We define the \emph{weight} of $T$ and the \emph{zeroth weight} of the entire graph $\Ga$ as
\begin{equation}
w(T)=\prod_{e\notin T}e,\quad w_0(\Ga)=\sum_{T\in \calB(\calM(\Ga))}w(T).
\label{eq:zerothweightJac}
\end{equation}
Here and henceforth we use the same letter to denote an edge $e\in E(\Ga)$ of a metric graph and its length. The weight $w_0(\Ga)$ is a homogeneous degree $g$ polynomial in the edge lengths $e$, and furthermore it has degree at most one in each edge length. By Remark~\ref{rem:dimension}, $I_0(\Jac(\Ga))$ has degree $g/2$ in the edge lengths, so we may expect it to be expressible in terms of $w_0(\Ga)$. This is indeed the case.

\begin{theorem}[Theorem 5.2 in~\cite{2014AnBakerKuperbergShokrieh}] The zeroth moment of the tropical Jacobian $\Jac(\Ga)$ of a metric graph $\Ga$ is equal to
\begin{equation}
I_0(\Jac(\Ga))=\sqrt{w_0(\Ga)}=\frac{w_0(\Ga)}{\sqrt{w_0(\Ga)}}.
\label{eq:I0Jac}
\end{equation}
\end{theorem}
Note that we have intentionally factored the zeroth moment as a fraction, so that the numerator is a homogeneous polynomial of degree $g$ in the edge lengths and the denominator (which happens to be the square root of the numerator) has degree $g/2$. 

We now derive a similar matroidal formula for the second moment, by reinterpreting the results of the papers~\cite{2023deJongShokrieh} and~\cite{2023RichmanShokriehWu} in terms of the graphic matroid $\calM(\Ga)$. Let
\[
\calI_k(\calM(\Ga))=\{F\in \calI(\calM(\Ga)):|F|=k\},\quad k=0,\ldots,n-1=\rk (\calM(\Ga))
\]
be the set of independent sets of $\calM(\Ga)$ having $k$ elements. The set $\calI_{n-1}(\calM(\Ga))=\calB(\calM(\Ga))$ is the set of bases of $\calM(\Ga)$, that is to say, the edge sets of the spanning trees of $\Ga$. Similarly, elements of $\calI_{n-2}(\calM(\Ga))$ are the edge sets of \emph{spanning 2-forests} of $\Ga$ (a spanning 2-forest is allowed to have an isolated vertex as one of its two trees).

In analogy with Equation~\eqref{eq:zerothweightJac}, we now define the \emph{weight} of a spanning $2$-forest $F\in \calI_{n-2}(\calM(\Ga))$ and the \emph{first weight} of the graph $\Ga$ as 
\begin{equation}
w(F)=\prod_{e\notin F}e,\quad w_1(\Ga)=\sum_{F\in \calI_{n-2}(\calM(\Ga))}w(F).
\end{equation}
In addition, we define \emph{length} of a spanning tree $T\in \calB(\calM(\Ga))$ and the \emph{length} of the graph $\Ga$ as
\[
\ell(T)=\sum_{e\notin T}e,\quad \ell(\Ga)=\sum_{e\in E(\Ga)}e.
\]
The quantities $\ell(T)$ and $\ell(\Ga)$ are linear in the edge lengths, while $w_1(\Ga)$ is a degree $g+1$ homogeneous polynomial in the edge lengths. By Remark~\ref{rem:dimension}, the second moment of $\Jac(\Ga)$ should be a function of degree $g/2+1$ in the edge lengths. An explicit formula for the second moment of the tropical Jacobian is one of the principal results of~\cite{2023deJongShokrieh} (see Theorem 8.1), and we now give a matroidal version of this formula.

\begin{proposition} The second moment of the tropical Jacobian $\Jac(\Ga)$ of a metric graph $\Ga$ is equal to
\begin{equation}
I_2(\Jac(\Ga))=\frac{p(\Ga)}{12\sqrt{w_0(\Ga)}},\quad p(\Ga)=2w_0(\Ga)\ell(\Ga)-\sum_{T\in \calB(\calM(\Ga))}w(T)\ell(T)-2w_1(\Ga).
\label{eq:I2Jac}
\end{equation}
\label{prop:I2Jac}
\end{proposition}

\begin{proof} Theorem 8.1 in~\cite{2023deJongShokrieh} gives a non-matroidal expression for $I_2(\Jac(\Ga))$. However, we can obtain the matroidal formula~\eqref{eq:I2Jac} by piecing together material from~\cite{2023deJongShokrieh} and the related paper~\cite{2023RichmanShokriehWu}. 
First, Theorem 11.4 in~\cite{2023deJongShokrieh} establishes the following relation between the zeroth and second moments of $\Jac(\Ga)$, the total length of $\Ga$, and a quantity called the $\tau$-invariant of $\Ga$:
\[
\frac{1}{2}\tau(\Ga)+\frac{I_2(\Jac(\Ga))}{I_0(\Jac(\Ga))}=\frac{1}{8}\ell(\Ga).
\]
We note that in~\cite{2023deJongShokrieh}, $I(\Jac(\Ga))$ denotes the quantity $I_2(\Jac(\Ga))/I_0(\Jac(\Ga))$, or the \emph{normalized} second moment of $\Ga$. Hence~\eqref{eq:I2Jac} is equivalent to the following formula for the $\tau$-invariant:
\begin{equation}
\tau(\Ga)=-\frac{1}{12}\ell(\Ga)+\frac{1}{6w_0(\Ga)}\sum_{T\in \calB(\calM(\Ga)))}w(T)\ell(T)+\frac{1}{3}\frac{w_1(\Ga)}{w_0(\Ga)}.
\label{eq:tau}
\end{equation}
First, assume that all edge lengths of $\Ga$ are equal to 1. Then $\ell(\Ga)=|E|$, while $w(T)=1$ and $\ell(T)=g$ for any $T\in \calB(\calM(\Ga))$, and~\eqref{eq:tau} specializes to
\[
\tau(\Ga)=-\frac{|E|}{12}+\frac{g}{6}+\frac{1}{3}\frac{w_1(\Ga)}{w_0(\Ga)}.
\]
This relation is proved in Remark 5.9 in~\cite{2023RichmanShokriehWu}, hence~$\eqref{eq:I2Jac}$ holds for $\Ga$.

In general, formula~\eqref{eq:I2Jac} can be proved using the arguments outlined in Section 5.3 of~\cite{2023RichmanShokriehWu}. Alternatively, we proceed as follows. If all edge lengths of $\Ga$ are integers, we pick the model having all edge lengths equal to 1, and it is elementary to verify that the right hand side of~\eqref{eq:I2Jac} is independent of the choice of model. Hence~\eqref{eq:I2Jac} holds for graphs with integer edge lengths. If the edge lengths of $\Ga$ are rational numbers, we can rescale by some $\lambda\in \ZZ$ to reduce to the case of integer edge lengths, and by virtue of Remark~\ref{rem:dimension} both sides of~\eqref{eq:I2Jac} scale by the same factor $\lambda^{g/2+1}$. Finally, $I_2(\Jac(\Ga))$ is clearly a continuous function of the edge lengths, hence the result follows for arbitrary real edge lengths.

\end{proof}

Before proceeding, we make several observations about the formula~\eqref{eq:I2Jac}. The numerator $p(\Ga)$ is a polynomial of degree $g+1$ in the edge lengths, consisting only of monomials of the form $x_{i_1}^2x_{i_2}\cdots x_{i_g}$ (the only type that occurs in the first sum) and $x_{i_1}x_{i_2}\cdots x_{i_{g+1}}$ (the only type that occurs in the second sum). If the matroid $\calM(\Ga)$ has already been computed, then~\eqref{eq:I2Jac} is an effective formula in terms of $\calM(\Ga)$. However, for an individual graph $\Ga$, finding $\calM(\Ga)$ and using~\eqref{eq:I2Jac} is computationally inefficient, as compared to using Theorem 8.1 in~\cite{2023deJongShokrieh}.

\subsection{The signed graphic matroid of a double cover} In this section, we recall how to associate a pair of dual matroids $\calM(\tGa/\Ga)$ and $\calM^*(\tGa/\Ga)$ to a double cover $\pi:\tGa\to \Ga$ of tropical curves, and recall a matroidal formula for the zeroth moment of the tropical Prym variety. In the case when $\pi$ is a free double cover, the matroid $\calM(\tGa/\Ga)$ is Zaslavsky's \emph{signed graphic matroid} as defined in~\cite{1982Zaslavsky}. We retain this terminology, even though our framework is somewhat more general than that of~\cite{1982Zaslavsky}. For a detailed description of the relation between Zaslavsky's matroid and double covers of graphs, not necessarily free, we refer the reader to~\cite{2023RoehrleZakharov}.

\begin{definition} Let $\pi:\tGa\to \Ga$ be a double cover of tropical curves. There exists a matroid $\calM^*(\tGa/\Ga)$, called the \emph{signed cographic matroid} of $\pi$, with the following properties.

\begin{enumerate} \item The ground set of $\calM^*(\tGa/\Ga)$ is the set $E_{\ud}(\Ga)=E(\Ga)\backslash E(\Ga_{\dil})$ of undilated edges of $\Ga$.

\item A set $F\subset E_{\ud}(\Ga)$ of undilated edges of $\Ga$ is independent for the matroid $\calM^*(\tGa/\Ga)$ if every connected component of $\Ga\backslash F$ has connected preimage in $\tGa$. 

\item A set $F\subset E_{\ud}(\Ga)$ of undilated edges of $\Ga$ is a basis for $\calM^*(\tGa/\Ga)$ if and only each connected component of $\Ga\backslash F$ is one of the following two types:
\begin{enumerate}
    \item A graph of genus one having no dilated vertices or edges, and having connected preimage in $\tGa$. 
    \item A tree having a connected dilation subgraph.

\end{enumerate}

\item The rank of $\calM^*(\tGa/\Ga)$ is equal to $t(\tGa/\Ga)=b_1(\tGa)-b_1(\Ga)=|E_{\ud}(\Ga)|-|V_{\ud}(\Ga)|$.

\end{enumerate}

\end{definition}

The \emph{signed graphic matroid} $\calM(\tGa/\Ga)$ of $\pi$ is the dual of $\calM^*(\tGa/\Ga)$, so the ground set of $\calM(\tGa/\Ga)$ is $E_{\ud}(\Ga)$ and the bases of $\calM(\tGa/\Ga)$ are the complements of the bases of $\calM^*(\tGa/\Ga)$. Following Zaslavsky's terminology, we call a subgraph $\Ga'\subset \Ga$ \emph{unbalanced} if it has connected preimage in $\tGa$, so basis for $\calM(\tGa/\Ga)$ is a maximal spanning subgraph whose connected components are unbalanced. The rank of $\calM(\tGa/\Ga)$ is equal to the number $v_{\ud}=|V_{\ud}(\Ga)|$ of undilated vertices of $\Ga$. As before, the set of edges of $\Ga$ is defined with respect to an underlying choice of model, and the matroids $\calM(\tGa/\Ga)$ and $\calM^*(\tGa/\Ga)$ depend on the model. We note that the vertex weights on $\tGa$ and $\Ga$ play no role in the matroids. 

We introduce one additional structure on the signed graphic matroid. 

\begin{definition} Let $\pi:\tGa\to \Ga$ be a double cover of tropical curves. The \emph{index} $i(F)$ of a basis $F\in \calB(\calM^*(\tGa/\Ga))$ of the signed cographic matroid is
\[
i(F)=4^{c(F)-1},
\]
where $c(F)$ is the number of connected components of $\Ga\backslash F$. The \emph{index} $i(F)$ of a basis $F\in \calB(\calM(\tGa/\Ga))$ of the signed graphic matroid is the index of $E_{\ud}(\Ga)\backslash F\in \calB(\calM^*(\tGa/\Ga))$.

\end{definition}

As we mentioned above, the Jacobian $\Jac(\Ga)$ of a metric graph $\Ga$ can be reconstructed from the graphic matroid $\calM(\Ga)$ and the edge length assignment $\ell:E(\Ga)\to \RR_{>0}$, and therefore all moments of $\Jac(\Ga)$ are computable from $\calM(\Ga)$ and $\ell$. Similarly, we showed in~\cite{2023RoehrleZakharov} that the Prym variety $\Prym(\tGa/\Ga)$ of a double cover $\tGa\to \Ga$ can be reconstructed from the signed graphic matroid $\calM(\tGa/\Ga)$, the edge length assignment $\ell:E(\Ga)\to \RR_{>0}$, and the index function $i:\calB(\calM(\tGa/\Ga))\to \ZZ$. Therefore, all moments of $\Prym(\tGa/\Ga)$ can be computed in terms of the triple $(\calM(\tGa/\Ga),\ell,i)$. 

In~\cite{2022LenZakharov} (see also~\cite{2024GhoshZakharov}), we found such a formula for the zeroth moment of $\Prym(\tGa/\Ga)$. In analogy with $\Jac(\Ga)$ and $\calM(\Ga)$, we define the \emph{weight} of a basis $F$ of $\calM(\tGa/\Ga)$ as the product of the lengths of the complementary undilated edges:
\begin{equation}
w(F)=\prod_{e\in E_{\ud}(\Ga)\backslash F} e.
\label{eq:ogodweight}
\end{equation}
The \emph{zeroth weight} of the double cover $\tGa\to \Ga$ is the sum of the weights of the bases, modified by the indices:
\begin{equation}
w_0(\tGa/\Ga)=\sum_{F\in \calB(\calM(\tGa/\Ga))} i(F)w(F).
\label{eq:zerothweightPrym}
\end{equation}

\begin{proposition} The zeroth moment of the continuous tropical Prym variety $\Prym_c(\tGa/\Ga)$ of a double cover $\pi:\tGa\to \Ga$ of tropical curves is equal to
\begin{equation}
I_0(\Prym_c(\tGa/\Ga))=\sqrt{w_0(\tGa/\Ga)}=\frac{w_0(\tGa/\Ga)}{\sqrt{w_0(\tGa/\Ga)}}.
\label{eq:I0Prym}
\end{equation}
\label{prop:I0Prym}
\end{proposition}

\begin{proof} For a free double cover, this result is Theorem 3.4 in~\cite{2022LenZakharov}. The result for arbitrary double covers can be obtained by edge contraction. Suppose that $\pi:\tGa\to\Ga$ is edge-free. To each dilated vertex $v\in V(\Ga)$ we attach a loop $e$, and replace the preimage vertex $\pi^{-1}(v)$ with a pair of vertices $\widetilde{v}^{\pm}$ connected by a pair of parallel edges $\widetilde{e}^{\pm}$ mapping to $e$. The result is a free double cover $\pi':\tGa'\to \Ga'$ whose contraction along the new loops is $\pi$. By Proposition~\ref{prop:Prymlimit}, the limit of $I_0(\tGa'/\Ga')$ is equal to $I_0(\tGa/\Ga)$. On the other hand, the signed graphic matroid $\calM(\tGa/\Ga)$ is the contraction of the matroid $\calM(\tGa'/\Ga')$ along the set of the new loops, hence $w_0(\tGa/\Ga)$ is the limit of $w_0(\tGa'/\Ga)$. This proves the result for edge-free double covers. Finally, for an arbitrary double cover, contracting the set of dilated edges does not change the zeroth moment by Proposition~\ref{prop:dilatededgesPrym}, and does not change the matroid (and hence $w_0(\tGa/\Ga)$) by definition. Hence the general case follows from the case of free double covers. 
    
\end{proof}

\section{The second moment of the tropical Prym variety in $g\leq 4$}

In this section, we state and prove our main result, which computes the second moment $I_2(\Prym(\tGa/\Ga))$ of the tropical Prym variety of any double cover $\tGa\to \Ga$ with $g(\Ga)\leq 4$. We will see that $I_2(\Prym(\tGa/\Ga))$ is a sum of two terms: a polynomial contribution $p(\tGa/\Ga)$ that is a direct generalization of the formula~\eqref{eq:I2Jac} for $I_2(\Jac(\Ga))$, with the graphic matroid $\calM(\Ga)$ replaced by the signed graphic matroid $\calM(\tGa/\Ga)$, and a piecewise-polynomial term $q(\tGa/\Ga)$ that does not appear for Jacobians. The presence of this term reflects the fact that the polyhedral type of the Voronoi polytope of $\Prym(\tGa/\Ga)$, which is the domain of integration for $I_2(\Prym(\tGa/\Ga))$, depends, in general, on the edge lengths of $\Ga$ (see Section~\ref{sec:PrymTorelli}). 

\begin{theorem} Let $\pi:\tGa\to \Ga$ be a double cover of tropical curves, where $g=g(\Ga)\leq 4$. The second moment of the tropical Prym variety $\Prym(\tGa/\Ga)$ is equal to
\begin{equation}
I_2(\Prym(\tGa/\Ga))=\frac{p(\tGa/\Ga)+q(\tGa/\Ga)}{12\sqrt{w_0(\tGa/\Ga)}}.
\label{eq:I2Prym}
\end{equation}
Here $p(\tGa/\Ga)$ is a degree $g$ polynomial function in the edge lengths of $\Ga$ given by Equation~\eqref{eq:pPrym}, and $q(\tGa/\Ga)$ is a degree $g$ piecewise-polynomial function in the edge lengths of $\Ga$ defined in Section~\ref{subsec:pp} and in the Appendix.
\label{thm:main}
\end{theorem}

\subsection{The polynomial term} We recall that the zeroth moment $I_0(\Prym(\tGa/\Ga))$ of the tropical Prym variety is a fraction~\eqref{eq:I0Prym} with a polynomial numerator, whose terms correspond to the bases of the signed graphic matroid $\calM(\tGa/\Ga)$, weighted by an appropriate index function. The formula~\eqref{eq:I2Jac} for the second moment $I_2(\Jac(\Ga))$ of a tropical Jacobian involves summation over the spanning trees (the bases of the graphic matroid $\calM(\Ga)$), the edges of $\Ga$, and the spanning 2-forests (independent sets of $\calM(\Ga)$ having one fewer element than a basis). We consider the analogous objects for the signed graphic matroid, and extend the index function appropriately.

\begin{definition} Let $\tGa\to \Ga$ be a double cover and let $e\in E_{\ud}(\Ga)$ be an undilated edge. The \emph{index} $i(e)$ is equal to
\[
i(e)=\begin{cases}
    4, & e\in F\mbox{ for some }F\in \calB(\calM^*(\tGa/\Ga))\mbox{ and }
    i(F)\geq 4\mbox{ for all }e\in F\in \calB(\calM^*(\tGa/\Ga)),\\
    1, & \mbox{otherwise.}
\end{cases}
\]
Equivalently, $i(e)=4$ if $e$ is a bridge and both connected components of $\Ga\backslash\{e\}$ are unbalanced, and $i(e)=1$ otherwise. We do not assign indices to dilated edges. 
    
\end{definition}


We recall that the rank of $\calM(\tGa/\Ga)$ is the number $v_{\ud}$ of undilated vertices of $\Ga$, and consider the set $\calI_{v_{\ud}-1}(\tGa/\Ga)$ of independent sets of $\calM(\tGa/\Ga)$ having one fewer element than a basis (the analogue of a 2-spanning forest for the signed graphic matroid). For $F\in\calI_{v_{\ud}-1}(\tGa/\Ga)$, we define the \emph{cut set} as
\[
\partial F=\{e\in E_{\ud}(\Ga)\backslash\{e\}:F\cup\{e\}\in \calB(\calM(\tGa/\Ga))\}.
\]
We now define the \emph{weight} and \emph{index} of $F\in\calI_{v_{\ud}-1}(\tGa/\Ga)$ as
\[
w(F)=\prod_{e\in E_{\ud}(\Ga)\backslash F}e,\quad
i(F)=\min_{e\in \partial F}i(F\cup\{e\})i(e),
\]
and define the polynomial $p(\tGa/\Ga)$ by analogy with $p(\Ga)$ (see~\eqref{eq:I2Jac}):
\begin{equation}
p(\tGa/\Ga)=2w_0(\tGa/\Ga)\sum_{e\in E_{\ud}(\Ga)}ei(e)-
\sum_{T\in \calB((\calM(\tGa/\Ga))}i(T)w(T)\sum_{e\in E_{\ud}(\Ga)\backslash T}ei(e)-
\sum_{F\in\calI_{v_{\ud}-1}(\tGa/\Ga)}i(F)w(F).
\label{eq:pPrym}
\end{equation}

\subsection{The piecewise-polynomial term} \label{subsec:pp} The presence of a piecewise-polynomial term $p(\tGa/\Ga)$ reflects the fact that the polyhedral type of the Voronoi polytope of $\Prym(\tGa/\Ga)$ depends, in general, on the edge lengths of $\Ga$. Furthermore, for a given double cover $\tGa\to \Ga$ the Voronoi polytope changes type if and only if the double cover $\tGa\to \Ga$ can be contracted to an $FS_n$-cover for some $n\geq 2$. Hence the form of $q(\tGa/\Ga)$ depends on the number of and relationship between such contractions. 

Let $\pi:\tGa\to \Ga$ be a double cover of tropical curves. A set of edges $\{e_1,\ldots,e_n\}\subset E_{\ud}(\Ga)$ is called an \emph{$FS_n$-set} if the double cover obtained by contracting all other edges of $\Ga$ is the $FS_n$-cover, in other words if the following conditions hold (cf.~Remark 6.2 in~\cite{2017Casalaina-MartinGrushevskyHulekLaza}):
\begin{enumerate}
    \item $\Ga\backslash \{e_1,\ldots,e_n\}$ consists of two unbalanced connected components.
    \item $\Ga\backslash F$ is connected for any proper subset $F\subset \{e_1,\ldots,e_n\}$.
\end{enumerate}
We note that this definition is purely matroidal: a set $F=\{e_1,\ldots,e_n\}$ is an $FS_n$-set if any basis of $\calM^*(\tGa/\Ga)$ containing $F$ has index at least 4 and if no proper subset of $F$ has this property. We observe that $n\leq g(\Ga)-1$. 

We now describe $q(\tGa/\Ga)$ for $g(\Ga)\leq 4$ in terms of the $FS_n$-sets of $\Ga$. To simplify the exposition, we give the formula for $q(\tGa/\Ga)$ only for free double covers; the formula for an arbitrary double cover may be obtained by appropriate edge contractions. 

\subsubsection*{$g(\Ga)=2$} There are no $FS_n$-sets for any $n\geq 2$, and
\[
q(\tGa/\Ga)=0.
\]

\subsubsection*{$g(\Ga)=3$} In this case $\Ga$ may have an $FS_2$-set. If there are none, then
\[
q(\tGa/\Ga)=0.
\]
If $\Ga$ has an $FS_2$-set $\{e,f\}\subset E_{\ud}(\Ga)$, it is necessarily unique and the graph $\Ga$ has the following form:
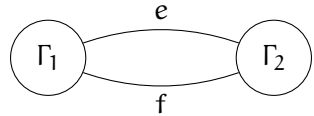
\begin{center}
\begin{tikzpicture}
    \draw[thin] (0,0) .. controls (1,0.5) and (2,0.5) .. (3,0);
    \draw[thin] (0,0) .. controls (1,-0.5) and (2,-0.5) .. (3,0);
    \node at (1.5,0.6) {$e$};
    \node at (1.5,-0.6) {$f$};
    \draw[fill,white] (0,0) circle(0.5);
    \draw[fill,white] (3,0) circle(0.5);
    \draw[thin] (0,0) circle(0.5);
    \draw[thin] (3,0) circle(0.5);
    \node at (0,0) {$\Ga_1$};
    \node at (3,0) {$\Ga_2$};
\end{tikzpicture}
\end{center}
Here $\Ga_1$ and $\Ga_2$ are unbalanced graphs of genus one, in other words, one of the following:
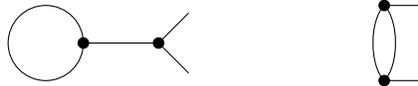
\begin{center}
    \begin{tikzpicture}
        \draw[thin] (0,0) -- (1,0);
        \draw[thin] (-0.5,0) circle(0.5);
        \draw[thin] (1.4,0.4) -- (1,0) -- (1.4,-0.4);
        \draw[fill] (0,0) circle(0.07);
        \draw[fill] (1,0) circle(0.07);
        \draw[thin] (4,0.5) .. controls (3.8,0.3) and (3.8,-0.3) .. (4,-0.5);
        \draw[thin] (4,0.5) .. controls (4.2,0.3) and (4.2,-0.3) .. (4,-0.5);
        \draw[thin] (4,0.5) -- (4.5,0.5);
        \draw[thin] (4,-0.5) -- (4.5,-0.5);
        \draw[fill] (4,0.5) circle(0.07);
        \draw[fill] (4,-0.5) circle(0.07);
    \end{tikzpicture}
\end{center}
The term $q(\tGa/\Ga)$ is given by the following piecewise-polynomial function:
\begin{equation}
q(\tGa/\Ga)=p_2(e,f)=\begin{cases}
    4e^2(3f-e),& e\leq f,\\
    4f^2(3e-f),& f\leq e.
\end{cases}
\label{eq:p2}
\end{equation}
The cone of edge lengths of $\Ga$ splits into two subcones $e\leq f$ and $f\leq e$, on each of which $q(\tGa/\Ga)$ is a polynomial.

\subsubsection*{$g(\Ga)=4$} In this case $\Ga$ may have up to three $FS_2$-sets and up to three $FS_3$-sets, and the piecewise-polynomial function $q(\tGa/\Ga)$ is given by progressively more complex expressions in terms of the $FS_n$-sets. To simplify the exposition, the explicit formulas for $q(\tGa/\Ga)$ are listed in the Appendix.

\subsection{Conjectural formulas} Motivated by Theorem~\ref{thm:main}, and the explicit formulas for $q(\tGa/\Ga)$ presented in the Appendix, we give conjectural formulas for the second moment of the tropical Prym variety for two classes of double covers. The bold reader is welcome to use these formulas to make further conjectures at their own peril.

\begin{conjecture} Let $\pi:\tGa\to \Ga$ be a double cover of tropical curves, and assume that $i(F)=1$ for all $F\in \calB(\tGa/\Ga)$, or equivalently that $\Ga$ has no $FS_n$-sets for any $n\geq 2$. The second moment of the tropical Prym variety $\Prym(\tGa/\Ga)$ is equal to
\[
I_2(\Prym(\tGa/\Ga))=\frac{p(\tGa/\Ga)}{12\sqrt{w_0(\tGa/\Ga)}},
\]
where $p(\tGa/\Ga)$ is given by Equation~\eqref{eq:pPrym}.
\end{conjecture}

\begin{conjecture} Let $\pi:\tGa\to \Ga$ be a double cover of tropical curves having a unique $FS_2$-set $\{e,f\}\subset \Ga$ and no $FS_n$-sets for any $n\geq 3$. Let $\Ga_1$ and $\Ga_2$ be the connected components of $\Ga\backslash\{e,f\}$ and let $\Ga'_1\to \Ga_1$ and $\Ga'_2\to \Ga_2$ be the associated double covers. The second moment of the tropical Prym variety $\Prym(\tGa/\Ga)$ is equal to
\[
I_2(\Prym(\tGa/\Ga))=\frac{p(\tGa/\Ga)+q(\tGa/\Ga)}{12\sqrt{w_0(\tGa/\Ga)}},
\]
where $p(\tGa/\Ga)$ is given by Equation~\eqref{eq:pPrym} and $q(\tGa/\Ga)$ is the piecewise-polynomial function
\[
q(\tGa/\Ga)=p_2(e,f)w_0(\tGa'_1/\Ga'_1)w_0(\tGa'_2/\Ga'_2).
\]
\end{conjecture}

\subsection{Proof of Theorem~\ref{thm:main}} We explain our calculations for $g=g(\Ga)=4$, the cases $g=2,3$ being substantially simpler. First, we note that, using the same reasoning as in Proposition~\ref{prop:I0Prym}, it is sufficient to prove Equation~\eqref{eq:I2Prym} for free double covers, since the general case reduces to the free case by edge contraction. 

Let $f:G\to T$ be a trigonal structure, where $G$ is a weighted graph of genus $4$, and let $G^{\st}$ be the stabilization of $G$. An assignment of edge lengths $\ell:E(T)\to \RR_{>0}$ promotes $f$ to a trigonal structure $\ph:\Ga\to \De$ of tropical curves, where the edge lengths of $\Ga$ and therefore $\Ga^{\st}$ are linear functions of the edge lengths of $\De$. We denote by $M_{G\to T}\subset \calM_4^{\trop}$ the set of metric graphs that are obtained by varying the edge lengths of $T$; this is a subcone of the cone  $M_{G^{\st}}$.

By Theorem~\ref{thm:CD}, every metric graph $\Ga'\in \calM_4$ admits a trigonal structure $\ph:\Ga\to \De$ (where $\Ga^{\st}=\Ga'$) and therefore lies in some cone $M_{G\to T}\subset \calM_4$. However, the combinatorial type of $\ph$ depends not only on the underlying graph $G'$ of $\Ga'$, but also on the relative values of the edge lengths. In other words, a cone $M_{G'}\subset \calM_4^{\trop}$, which parametrizes metric graphs with underlying graph $G'$, cannot in general be covered by a single cone $M_{G\to T}$, which parametrizes metric graphs together with a fixed trigonal structure.

We now fix a trigonal structure $f:G\to T$ and choose a connected free double cover $p:\tG\to G$. Varying the edge lengths of $T$, we obtain a cone $R_{\tG\to G\to T}\subset R_{\tG/G}$ in the moduli space of double covers $\calR_4^{\trop}$, parametrizing towers of the form $\tGa\to \Ga\to \De$. For any double cover $\pi:\tGa\to \Ga$ in $R_{\tG\to G\to T}$, we apply the trigonal construction (Theorem~\ref{thm:trigonal}) and obtain a tetragonal structure $\ph:\Pi\to \De$ such that $\Prym(\tGa/\Ga)$ is isomorphic to $\Jac(\Pi)$. We now compute the second moment  
\[
I_2(\Prym(\tGa/\Ga))=I_2(\Jac(\Pi))
\]
by using the matroidal formula in Proposition~\ref{prop:I2Jac}. Since the right hand side of~\eqref{eq:I2Jac} is polynomial in the edge lengths of $\Pi$, and since the latter are linear functions of the edge lengths of $\De$ and therefore those of $\Ga$, we obtain a polynomial expression for $I_2(\Prym(\tGa/\Ga))$ on the cone $R_{\tG\to G\to T}$. However, a cone $R_{\tG/G}\subset \calR_4^{\trop}$ cannot in general be covered by a single cone $R_{\tG\to G\to T}$, therefore the second moment is in general piecewise-polynomial on $R_{\tG/G}$.

We now use this method to verify Equation~\eqref{eq:I2Prym} by enumerating all towers $\tG\to G\to T$, where $p:\tG\to G$ is a free double cover of a graph $G$ of genus 4 and $f:G\to T$ is a trigonal structure. While this may seem to be a daunting task, it is sufficient to restrict our attention to double covers over trigonal structures satisfying the following properties, which we call \emph{generic}:

\begin{enumerate} \item The target tree $T$ has $2g+2n-5=9$ edges and furthermore is \emph{trivalent}, meaning $\val(v)\leq 3$ for all $v\in v(T)$.

\item The stabilization $G^{\st}$ of $G$ has $|E(G^{\st})|=3g-3=9$ edges.

\item The edge lengths of $G^{\st}$ (induced by $f$) are linearly independent functions of the edge lengths of $T$.
    
\end{enumerate}

We claim that the union of the cones $M_{G\to T}\subset \calM_4$ corresponding to generic trigonal structures is a dense subset of $\calM_4$, and that therefore the cones $M_{\tG\to G\to T}$ corresponding to free double covers of generic trigonal structures form a dense subset in the moduli space of free double covers. Since both sides of Equation~\eqref{eq:I2Prym} are continuous, it is therefore sufficient to verify the formula only for such double covers.

Let $f:G\to T$ be a trigonal structure. By Proposition~\ref{prop:2g+2n-5}, we can assume that $T$ has at most $2g+2n-5=9$ edges. If $T$ has fewer edges, or if the edge lengths of $G^{\st}$ are not linearly independent, then the locus $M_{G\to T}$ has dimension strictly less than $\dim \calM_4=3g-3=9$. Now suppose that $T$ is not trivalent. Let $a_i$ denote, as in the proof Proposition~\ref{prop:2g+2n-5}, the number of vertices of $T$ of valence $a_i$. Recall that we showed that
\[
\Ram(f)=2g+2n-2=12\geq 2a_1+a_2.
\]
If $a_i\geq 1$ for some $i\geq 4$, then
\[
|E(T)|=3|V(T)|-2|E(T)|-3=3(a_1+a_2+a_3+\cdots)-(a_1+2a_2+3a_3+\cdots)-3<2a_1+a_2-3\leq 9.
\]
Therefore a cone $M_{G\to T}\subset \calM_4$ corresponding to a non-generic trigonal structure has positive codimension. Since there are only finitely many such cones, it follows that the generic cones are dense in $\calM_4$.

We now outline how to enumerate all free double covers $\tG\to G\to T$ over generic trigonal structures, and how to verify~\eqref{eq:I2Prym} on the corresponding cones $R_{\tG\to G\to T}\subset \calR_4^{\trop}$.

\subsubsection*{List of trees} The first step is to create a list $\calT$ of trivalent trees with $9$ edges. There are 37 such trees, and each tree is oriented by ordering its vertices. 

\subsubsection*{Type markings} Each $T\in \calT$ is the target of a collection $f:G\to T$ of trigonal structures, which we encode using appropriate markings on $T$. We mark the edges and vertices of $T$ as \emph{type I}, \emph{type II}, or \emph{type III}, where the type of an edge or vertex is the number of its preimages in $G$. We then label the  preimages of the vertices as follows:
\begin{enumerate} \item If $v\in V(T)$ has type I, then $f^{-1}(v)=\{\tv_1\}$ with $d_f(\tv_1)=3$.
\item If $v\in V(T)$ has type II, then $f^{-1}(v)=\{\tv_1,\tv_2\}$ with $d_f(v_1)=1$ and $d_f(v_2)=2$.
\item If $v\in V(T)$ has type III, $f^{-1}(v)=\{\tv_1,\tv_2,\tv_3\}$ with $d_f(\tv_1)=d_f(\tv_2)=d_f(\tv_3)=1$.

\end{enumerate}
For coding purposes, it is convenient to introduce redundant labels $\tv_1=\tv_2=\tv_3$ for the preimages of a type I vertex and $\tv_2=\tv_3$ for the preimages of a type II vertex. We label the edges of $G$ similarly.

The number $3^{|E(T)|+|V(T)|}=3^{19}=1162261467$ of possible type markings of a given $T\in \calT$ is quite large, so we first derive a number of admissibility criteria. First, it is clear that the type of an edge is not less than the types of its root vertices. Now let $f:G\to T$ be a trigonal structure, then
\[
\Ram(f)=\deg(f)\chi(T)-\chi(G)=2g+4=12
\]
by the global Riemann--Hurwitz formula~\eqref{eq:globalRH}. We consider how the ramification degree is distributed among the vertices of $G$.

Let $v\in V(T)$ be an extremal vertex at the end of an edge $e\in E(T)$. We saw in the proof of Proposition~\ref{prop:2g+2n-5} that if the ramification degree of $f$ at each $\tv\in f^{-1}(v)$ is less than 2, then all edges of $G$ over $e$ are also extremal. If we equip $T$ with edge lengths, then the edge lengths of the stabilization of $G$ will not depend on $e$, and the resulting cone in $\calM_4$, and therefore $\calR_4$, will not have the maximum dimension. Similarly, let $v\in V(T)$ have $\val(v)=2$ and let $e_1,e_2\in E(T)$ be the edges at $v$. If $f$ is unramified at each $\tv\in f^{-1}(v)$, then all these $\tv$ also have valence 2 and can be removed once edge lengths are added, so the edge lengths of $G$ depend only on $e_1+e_2$ and the resulting locus is not of maximum dimension. 

We conclude, as in Proposition~\ref{prop:2g+2n-5}, that the total ramification degree is at least 2 over each $v\in V(T)$ with $\val(v)=2$ and 1 over each $v\in V(T)$ with $\val(v)=1$. Let $a_1, a_2, a_3$ be the number of vertices of $T$ of valency 1, 2, and 3, respectively, then $\Ram(f)\geq 2a_1+a_2$. On the other hand,
\[
a_1+a_2+a_3=|V(\Delta)|=2g+2=10,\quad a_1+2a_2+3a_3=4g+2=18,
\]
therefore $2a_1+a_2=2g+4=12=\Ram(f)$. The only way that this is possible is if the following conditions hold:

\begin{enumerate}
    \item Over each $v\in V(T)$ of valence 1 there is a unique vertex $\tv\in V(G)$ with $\Ram_f(\tv)=2$.
    \item Over each $v\in V(T)$ of valence 2 there is a unique vertex $\tv\in V(G)$ with $\Ram_f(\tv)=1$.
    \item $\Ram_f(\tv)=0$ at all other $\tv\in V(G)$. 

\end{enumerate}

An exhaustive enumeration yields the possible type combinations $(a; b_i)$ of a vertex and its attached edges:

\begin{enumerate}
    \item If $v$ is an extremal vertex on an edge $e$, then the possible types of $v$ and $e$ are $(\mathrm{I};\mathrm{I})$ and $(\mathrm{II};\mathrm{III})$. We note that if $v$ and $e$ are both type I, then $e$ and its image in $T$ can be contracted, hence the resulting cone will not have maximal dimension. 

    \item If $v$ lies on edges $e_1$ and $e_2$, then the possible types of $v$, $e_1$, and $e_2$ are $(\mathrm{I};\mathrm{I},\mathrm{II})$ and $(\mathrm{II};\mathrm{II},\mathrm{III})$.

    \item If $v$ is trivalent with root edges $e_1$, $e_2$, and $e_3$, then their possible types are $(\mathrm{I};\mathrm{I},\mathrm{I},\mathrm{III})$, $(\mathrm{I};\mathrm{I},\mathrm{II},\mathrm{II})$, $(\mathrm{II};\mathrm{II},\mathrm{II},\mathrm{III})$, and $(\mathrm{III};\mathrm{III},\mathrm{III},\mathrm{III})$. 
    
\end{enumerate}

We observe that the edge type markings on $T$ are highly constrained (for example, there are only four possible combinations at a trivalent vertex), and furthermore the vertex types are determined by the edge types. Going through the $3^9=19683$ possible edge type markings for each $T\in \calT$, we obtain a list $\calT'$ of $1184$ trees labeled with valid vertex and edge types. The trees in $\calT$ generally have large automorphism groups, and we do not check for redundant labelings.

\subsubsection*{Monodromy labels} Let $T\in \calT'$ be a tree with vertex and edge labels. To define a trigonal structure $g:G\to T$, we choose \emph{monodromy labels} $\sigma_e\in S_3$ for all $e\in E(T)$. We recall that the vertices $\{\tv_1,\tv_2,\tv_3\}$ over a given $v\in V(T)$ are reduntantly labeled, so $\tv_2=\tv_3$ if $v$ has type II and all vertices are the same if $v$ has type I; and the edges are labeled similarly. We now construct $G$ by attaching the oriented edges $\te_i$ over each $e\in E(T)$ using the monodromy labels, as if $f:G\to T$ were a free cover of degree 3:
\[
s(\te_i)=\widetilde{s(e)}_i,\quad t(\te_i)=\widetilde{s(e)}_{\sigma_e(i)}.
\]
The possible monodromy labels are strongly constrained by the types of $e$, $s(e)$, and $t(e)$, so we do not need to enumerate $6^{|E(T)|}=6^9=10077696$ choices. For example, $g_e=1$ is the only choice if either of the root vertices $s(e)$ or $t(e)$ has type I, and all six elements of $S_3$ are used only if $e$, $s(e)$, and $t(e)$ all have type III. At a type III trivalent vertex $v$ with three type III edges, we may rearrange the labels of $f^{-1}(v)$ and choose one of the monodromy elements of the incident edges to be trivial. The result is a list $\calT''$ of $12977$ trees $T$ with type and monodromy labels; again we do not take automorphisms of $T$ into account. 

\subsubsection*{Dimension count} Each $T\in \calT''$ is equipped with enough decorations to define a trigonal structure $f:G\to T$, where $g(G)=4$. We retain only those that are generic (and also check $G$ for connectedness), and obtain a list $\mathrm{Trig}_4$ of 821 generic trigonal structures of genus 4 graphs, so that the union of the corresponding cones $M_{G\to T}$ is dense in $\calM_4$. Many of these structures are in fact isomorphic, and we do not check for this. 

\subsubsection*{Double covers} For each trigonal structure $f:G\to T$ in $\mathrm{Trig}_4^{\mathrm{gen}}$ we now enumerate all $2^g-1=15$ connected free double covers $p:\tG\to G$ by choosing a spanning tree of $G$ and enumerating the nontrivial sign assignments on the complementary edges. We obtain a list $\mathrm{Trig}^2_4$ of $821\times 15=12315$ connected free double covers $\tG\to G\to T$ of generic trigonal graphs. 

\medskip
For each double cover $\tG\to G\to T$ in $\mathrm{Trig}^2_4$, we now directly verify Equation~\eqref{eq:I2Prym} by computing the left hand side as the second moment of a Jacobian, using Proposition~\ref{prop:I2Jac}, and directly evaluating the right hand side. This proves the theorem. The entire calculation and verification takes less than an hour to complete on a 2020 MacBook Air.

\section{Extension of the Prym--Torelli morphism} In this brief section, we explain why $I_2(\Prym(\tGa/\Ga))$, unlike $I_2(\Jac(\Ga))$, has a piecewise-polynomial term, and state some actual and conjectural geometric applications of the calculations of Theorem~\ref{thm:main}. We refer the reader to~\cite{2017Casalaina-MartinGrushevskyHulekLaza} for the necessary background.

We consider the following moduli spaces together with their compactifications:
\begin{itemize}
    \item $\calM_g\subset \overline{\calM}_g$, the moduli space of smooth genus $g$ curves compactified by the Deligne--Mumford moduli space of stable genus $g$ curves.
    \item $\calR_g\subset \overline{\calR}_g$, the moduli space of connected \'etale double covers of genus $g$ curves compactified by the space of admissible double covers of stable curves.
    \item $\calA_g\subset \overline{\calA}^V_g$, the moduli space of ppavs of dimension $g$ and its second Voronoi compactification. 
\end{itemize}
The assignments $C\mapsto \Jac(C)$ and $(C'\to C)\mapsto \Prym(C'/C)$ define the \emph{Torelli} and \emph{Prym--Torelli morphisms}
\[
t_g:\calM_g\to \calA_g,\quad p_g:\calR_g\to\calA_{g-1}^,
\]
and a natural question is to ask whether $t_g$ and $p_g$ extend to the boundary. 

Mumford and Namikawa~\cite{1980Namikawa} proved that $t_g$ extends to a morphism $\overline{t}_g:\overline{\calM}_g\to \overline{\calA}^V_g$, and we describe the combinatorial result that underlies this theorem. The boundary strata of $\overline{\calM}_g$ are indexed by stable weighted graphs of genus $G$:
\[
\overline{\calM}_g=\bigsqcup_{G}\calM_G.
\]
For a stable weighted graph $G$ let $M_G=\RR_{>0}^{|E(G)|}$ be the cone of metric graphs having underlying model $G$, and for each $\Gamma\in M_G$ construct the Voronoi polytope of the Jacobian $\Jac(\Ga)$. The morphism  $t_g$ extends to the stratum $\calM_G$ (the target being the second Voronoi compactification) if and only if the following property holds: the Voronoi polytopes $\Vor(\Jac(\Ga))$ are normally equivalent for all $\Ga\in M_G$. It turns out that this is indeed the case for all stable weighted graphs $G$, and that therefore $t_g$ extends over all of $\overline{\calM}_g$. 

Friedman and Smith showed~\cite{1986FriedmanSmith} that the Prym--Torelli morphism $p_g$ does not extend to the entire boundary of $\overline{\calR}_g$, and the indeterminancy locus of $\overline{p}_g:\overline{\calR}_g\dashrightarrow \overline{\calA}^V_{g-1}$ was identified in~\cite{2002AlexeevBirkenhakeHulek}. We describe the analogous combinatorial problem. The boundary strata of $\overline{\calR}_g$ are indexed by double covers of stable weighted graphs:
\[
\overline{\calR}_g=\bigsqcup_{\tG\to G}\calR_{\tG/G}.
\]
For a double cover $p:\tG\to G$, compute the Voronoi polytopes of the continuous Prym varieties $\Vor(\Prym_c(\tGa/\Ga))$ for all double covers of metric graphs $(\tGa\to \Ga)\in R_{\tG/G}$ having underlying model $p$. Then $p_g$ extends to a locus $\calR_{\tG/G}\subset \overline{\calR}_g$ if and only if all these Voronoi polytopes are normally equivalent. It turns out that this is not the case in general, and that this property fails precisely when $p:\tG\to G$ is a degeneration of an $FS_n$-double cover for some $n\geq 2$  (meaning that the latter is an edge contraction of the former). Hence the locus of indeterminancy of $\overline{p}_g:\overline{\calR}_g\dashrightarrow \overline{\calA}^V_{g-1}$ is the union of the closures of the strata $\calR_{FS_n}$ for $n\geq 2$. We note that the piecewise-polynomial term $q(\tGa/\Ga)$ in the formula for the second moment $I_2(\Prym(\tGa/\Ga))$ occurs, in genus 3 and 4, exactly on these cones. 

A natural question is to find a resolution of $\overline{p}_g:\overline{\calR}_g\dashrightarrow \overline{\calA}^V_{g-1}$ via a sequence of toric blowups. The relevant combinatorial problem is the following: for each cone $R_{\tG/G}$, where $p:\tG\to G$ is a degeneration of an $FS_n$-cover, find a decomposition into subcones such that, on each subcone, the Voronoi polytope has constant normal fan. The minimal such decomposition, which we call the \emph{Voronoi decomposition}, describes the minimal blowup. This problem has been solved up to codimension 3 in~\cite{2002AlexeevBirkenhakeHulek} and in the companion paper~\cite{2002Vologodsky}. 

We claim that the collection of cones $R_{\tG\to G\to T}\subset R_{\tG/\G}$ constructed in Theorem~\ref{thm:main} solve this problem in $g=4$ for all double covers $\tG\to G$, and hence describe a full resolution of the Prym--Torelli map $\overline{p}_4:\overline{\calR}_4\dashrightarrow \overline{\calA}^V_3$. Indeed, for any double cover $\tGa\to \Ga$ in $R_{\tG\to G\to T}$, where the target graph $G$ is equipped with a fixed trigonal structure $G\to T$, the Prym variety $\Prym(\tGa/\Ga)$ is isomorphic to a Jacobian $\Jac(\Pi)$. Varying the edge lengths inside the subcone changes the edge lengths of $\Pi$ but not the underlying graph. Therefore, the Voronoi polytope of $\Prym(\tGa/\Ga)$ has constant normal cone on the subcone $R_{\tG\to G\to T}$, because by Mumford and Namikawa's result this statement holds for the Jacobian $\Jac(\Pi)$. We note that we construct the decomposition of the cone $R_{\tG/\G}$ into subcones of the form $R_{\tG\to G\to T}$ only for maximal-dimensional cones (corresponding to points on the boundary of $\overline{\calR}_g$), but this decomposition can be extended to all cones by edge contraction.

Unfortunately, it turns out that the cone decomposition described above is, in general, much finer that the Voronoi decomposition. For example, it is nontrivial on many cones $R_{\tG/G}$ corresponding to double covers $\tG\to G$ that do not degenerate to $FS_n$ for any $n\geq 2$, for which we know the Voronoi decomposition to be trivial. In other words, the blowups determined by the cones $R_{\tG\to G\to T}$ are not minimal, and indeed blow up strata that lie outside the indeterminancy locus. We end by conjecturing how the Voronoi decomposition can be determined from our calculations.

\begin{conjecture} The Voronoi decomposition of a cone $R_{\tG/G}$ is given by the domains of polynomiality of the second moment of the tropical Prym variety $I_2(\Prym(\tGa/\Ga))$ as $\tGa\to \Ga$ varies over $R_{\tG/G}$.
 
\end{conjecture}

\label{sec:PrymTorelli}

\appendix

\section{The piecewise-polynomial term $q(\tGa/\Ga)$ for $g(\Ga)=4$}

Let $\pi:\tGa\to \Ga$ be a free double cover of a stable trivalent graph of genus $4$. The piecewise-polynomial term $q(\tGa/\Ga)$ depends on the number of $FS_2$- and $FS_3$-sets of $\Ga$, in other words, on the number of ways of contracting $\pi$ to an $FS_2$- or an $FS_3$-cover. We list the formulas for $q(\tGa/\Ga)$ in order of increasing number of $FS_2$-sets. As we shall see, these formulas are built from a number of simple blocks, and we recall the piecewise-polynomial function~\eqref{eq:p2}
\[
p_2(e,f)=\begin{cases}
    4e^2(3f-e),& e\leq f,\\
    4f^2(3e-f),& f\leq e,
\end{cases}
\]
that gives the formula for $q(\tGa/\Ga)$ in the $g(\Ga)=3$ case.

\subsubsection*{No $FS_2$-sets, no $FS_3$-sets} In this case we set
\[
q(\tGa/\Ga)=0.
\]

\subsubsection*{No $FS_2$-sets, one $FS_3$-set} Let $\{e,f,g\}\subset E(\Ga)$ be the unique $FS_3$-set. There is a unique underlying graph $\Ga$ with this property:
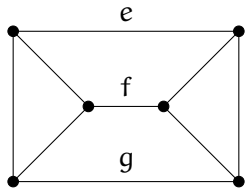
\begin{center}
\begin{tikzpicture}
    \draw[thin] (0,0) -- (3,0) -- (3,2) -- (0,2) -- (0,0);
    \draw[thin] (0,0) -- (1,1) -- (2,1) -- (3,0);
    \draw[thin] (0,2) -- (1,1);
    \draw[thin] (2,1) -- (3,2);
    \draw[fill] (0,0) circle(0.07);
    \draw[fill] (3,0) circle(0.07);
    \draw[fill] (3,2) circle(0.07);
    \draw[fill] (0,2) circle(0.07);
    \draw[fill] (1,1) circle(0.07);
    \draw[fill] (2,1) circle(0.07);
    \node[above] at (1.5,2) {$e$};
    \node[above] at (1.5,1) {$f$};
    \node[above] at (1.5,0) {$g$};
\end{tikzpicture}
\end{center}
There are a number of possible signed graph structures on $\Ga$, but each of the two triangles must be unbalanced. The term $q(\tGa/\Ga)$ is given by the following piecewise-polynomial function:
\begin{equation}
    q(\tGa/\Ga)=p_3(e,f,g)=\begin{cases}
        2e^2(e^2-2ef-2eg+6fg), & e\leq \min(f,g), \\
        2f^2(f^2-2ef-2fg+6eg), & f\leq \min(e,g), \\
        2g^2(g^2-2eg-2fg+6ef), & g\leq \min(e,f).
    \end{cases}
    \label{eq:p3}
\end{equation}
The cone of edge lengths of $\Ga$ splits into three domains of polynomiality, according to which of the three edges $e,f,g$ is the shortest.

\subsubsection*{One $FS_2$-set, no $FS_3$-sets} Let $\{e,f\}\subset E(\Ga)$ be the unique $FS_2$-set. Then $\Ga\backslash\{e,f\}=\Ga_1\sqcup \Ga_2$, where the graphs $\Ga_1$ and $\Ga_2$ are unbalanced, and we assume without loss of generality that $g(\Ga_1)=2$ and $g(\Ga_2)=1$:
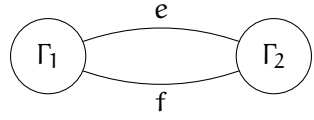
\begin{center}
\begin{tikzpicture}
    \draw[thin] (0,0) .. controls (1,0.5) and (2,0.5) .. (3,0);
    \draw[thin] (0,0) .. controls (1,-0.5) and (2,-0.5) .. (3,0);
    \node at (1.5,0.6) {$e$};
    \node at (1.5,-0.6) {$f$};
    \draw[fill,white] (0,0) circle(0.5);
    \draw[fill,white] (3,0) circle(0.5);
    \draw[thin] (0,0) circle(0.5);
    \draw[thin] (3,0) circle(0.5);
    \node at (0,0) {$\Ga_1$};
    \node at (3,0) {$\Ga_2$};
\end{tikzpicture}
\end{center}
Then
\[
q(\tGa/\Ga)=p_2(e,f)w_0(\tGa_1/\Ga_1),
\]
where $\tGa_1\to \Ga_1$ is the restriction of the double cover to $\Ga_1$ and $w_0(\tGa_1/\Ga_1)$ is the zeroth moment of the Prym variety of the double cover $\tGa_1\to \Ga_1$ (see Eq.~\eqref{eq:zerothweightPrym}). The domains of polynomiality are $e\leq f$ and $f\leq e$.

\subsubsection*{One $FS_2$-set, one $FS_3$-set} The two sets always have exactly one edge in common, so we label them $\{e,f\}$ and $\{e,g,h\}$. The graph $\Ga$ has the following form:
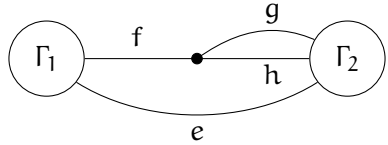
\begin{center}
\begin{tikzpicture}
    \draw[thin] (0,0) -- (4,0);
    \draw[thin] (2,0) .. controls (2.7,0.5) and (3.3,0.5) .. (4,0);
    \draw[thin] (0,0) .. controls (1,-1) and (3,-1) .. (4,0);
    \node at (1.2,0.3) {$f$};
    \node at (3,0.6) {$g$};
    \node at (3,-0.2) {$h$};
    \node at (2,-1) {$e$};
    \draw[fill] (2,0) circle(0.07);
    \draw[fill,white] (0,0) circle(0.5);
    \draw[thin] (0,0) circle(0.5);
    \node at (0,0) {$\Ga_1$};
    \draw[fill,white] (4,0) circle(0.5);
    \draw[thin] (4,0) circle(0.5);
    \node at (4,0) {$\Ga_2$};    
\end{tikzpicture}
\end{center}
The subgraphs $\Ga_1$ and $\Ga_2$ are unbalanced graphs of genus one. There are four domains of polynomiality. If $e\leq f$, then
\[
q(\tGa/\Ga)=p_2(e,f)w_0(\tGa'_2/\Ga'_2)+12e^2gh,
\]
where $\tGa'_2\to\Ga'_2$ is the double cover over the graph $\Ga'_2=\Ga_2\cup\{g,h\}$.  If $f\leq e$, then
\[
q(\tGa/\Ga)=p_3(e-f,g,h)+p_2(f,e)w_0(\tGa_2'/\Ga_2')+12fgh(2e-f),
\]
and there are three domains of polynomiality, depending on which of the three quantities $e-f$, $g$, and $h$ is the smallest.

\subsubsection*{One $FS_2$-sets, two $FS_3$-sets} The $FS_3$-sets are always disjoint, and each contains exactly one element from the $FS_2$-set, so we can denote the sets by $\{e_1,e_2\}$, $\{e_1,f_1,g_1\}$, and $\{e_1,f_1,g_1\}$. The graph $\Gamma$ has the following form:
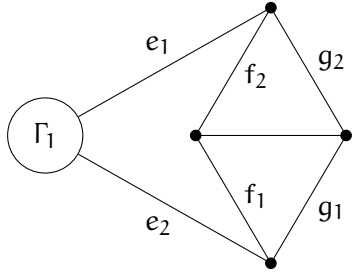
\begin{center}
\begin{tikzpicture}
    \draw[thin] (0,0) -- (3,1.7) -- (2,0) -- (3,-1.7) -- (0,0);
    \draw[thin] (3,1.7) -- (4,0) -- (3,-1.7);
    \draw[thin] (2,0) -- (4,0);
    \draw[fill] (3,1.7) circle(0.07);
    \draw[fill] (2,0) circle(0.07);
    \draw[fill] (3,-1.7) circle(0.07);
    \draw[fill] (4,0) circle(0.07);
    \draw[fill,white] (0,0) circle(0.5);
    \draw[thin] (0,0) circle(0.5);
    \node at (0,0) {$\Ga_1$};
    \node at (1.5,-1.2) {$e_2$};
    \node at (1.5,1.2) {$e_1$};
    \node at (2.8,0.8) {$f_2$};
    \node at (2.8,-0.8) {$f_1$};
    \node at (3.8,1) {$g_2$};
    \node at (3.8,-1) {$g_1$};
\end{tikzpicture}
\end{center}
The subgraph $\Ga_1$ is an unbalanced graph of genus one, and both triangles on the right are unbalanced. We denote by $\tGa'\to \Ga'$ the double cover over the subgraph $\Ga'\subset \Ga$ consisting of the two triangles. There are six domains of polynomiality. In the cone $e_1\leq e_2$ we have
\[
f_{pp}(\tGa/\Ga)=p_3(e_2-e_1,f_2,g_2)+p_2(e_1,e_2)w_0(\tGa'/\Ga')-12e_1^2f_2g_2+24e_1e_2f_2g_2+12e_1^2f_1g_1,
\]
and the domains of polynomiality are determined by which of $e_2-e_1$, $f_2$, and $g_2$ is the smallest. Exchanging $1$ and $2$ everywhere, we obtain the formula for $p_2(\tGa/\Ga)$ in the cone $e_2\leq e_1$.

\subsubsection*{Two $FS_2$-sets, no $FS_3$-sets} There are two subcases consider. The $FS_2$-sets $\{e_1,f_1\}$ and $\{e_2,f_2\}$ may be disjoint, in which case $\Ga$ has the following form:
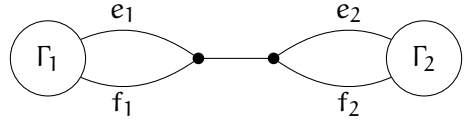
\begin{center}
\begin{tikzpicture}
    \draw[thin] (0,0) .. controls (0.7,0.5) and (1.3,0.5) .. (2,0);
    \draw[thin] (0,0) .. controls (0.7,-0.5) and (1.3,-0.5) .. (2,0);
    \node at (1,0.6) {$e_1$};
    \node at (1,-0.6) {$f_1$};
    \draw[fill,white] (0,0) circle(0.5);
    \draw[thin] (0,0) circle(0.5);
    \node at (0,0) {$\Ga_1$};
    \draw[fill] (2,0) circle(0.07);
    \draw[thin] (2,0) -- (3,0);
    \draw[fill] (3,0) circle(0.07);
    \draw[thin] (3,0) .. controls (3.7,0.5) and (4.3,0.5) .. (5,0);
    \draw[thin] (3,0) .. controls (3.7,-0.5) and (4.3,-0.5) .. (5,0);
    \node at (4,0.6) {$e_2$};
    \node at (4,-0.6) {$f_2$};
    \draw[fill,white] (5,0) circle(0.5);
    \draw[thin] (5,0) circle(0.5);
    \node at (5,0) {$\Ga_2$};
\end{tikzpicture}
\end{center}
Here the subgraphs $\Ga_1$ and $\Ga_2$ are unbalanced graphs of genus one. Denote by $\tGa'_1\to \Ga'_1$ and $\tGa'_2\to \Ga'_2$ the double covers over the genus two subgraphs $\Ga'_1=\Ga_1\cup\{e_1,f_1\}$ and $\Ga'_2=\Ga_2\cup\{e_2,f_2\}$. Then
\[
q(\tGa/\Ga)=p_2(e_1,f_1)w_0(\tGa'_2/\Ga'_2)+p_2(e_2,f_2)w_0(\tGa'_1/\Ga'_1)+24e_1e_2f_1f_2,
\]
and there are four domains of polynomiality.

Alternatively, the $FS_2$-sets $\{e,f\}$ and $\{e,g\}$ may have a common edge, in which case $\Ga$ has the following form:
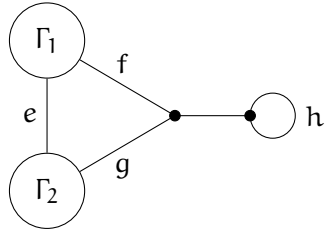
\begin{center}
\begin{tikzpicture}
    \draw[thin] (0,1) -- (0,-1) -- (1.7,0) -- (0,1);
    \draw[fill,white] (0,1) circle(0.5);
    \draw[thin] (0,1) circle(0.5);
    \node at (0,1) {$\Ga_1$};
    \draw[fill,white] (0,-1) circle(0.5);
    \draw[thin] (0,-1) circle(0.5);
    \node at (0,-1) {$\Ga_2$};
    \node[left] at (0,0) {$e$};
    \node at (1,0.7) {$f$};
    \node at (1,-0.7) {$g$};
    \draw[fill] (1.7,0) circle(0.07);
    \draw[thin] (1.7,0) -- (2.7,0);
    \draw[fill] (2.7,0) circle(0.07);
    \draw[thin] (3,0) circle(0.3);
    \node[right] at (3.3,0) {$h$};
\end{tikzpicture}
\end{center}
The subgraphs $\Ga_1$ and $\Ga_2$ are unbalanced graphs of genus one and $h$ is an even loop. We have
\[
q(\tGa/\Ga)=p_2(e,f+g)h,
\]
and there are two domains of polynomiality. We note that $h=w_0(\tGa'/\Ga')$, where $\tGa'\to \Ga'$ is the double cover over the graph $\Ga'$ obtained from $\Ga$ by removing either $\Ga_1$ or $\Ga_2$.

\subsubsection*{Two $FS_2$-sets, one $FS_3$-set} The three sets have a unique common edge and are otherwise disjoint, so we can label them $\{e,f_1\}$, $\{e,f_2\}$, and $\{e,g,h\}$. The graph $\Ga$ has the following form:
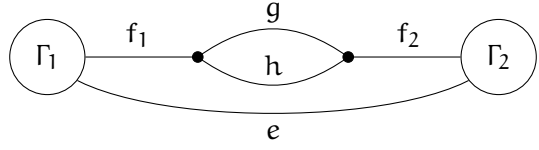
\begin{center}
\begin{tikzpicture}
    \draw[thin] (0,0) -- (2,0);
    \draw[thin] (2,0) .. controls (2.7,0.5) and (3.3,0.5) .. (4,0);
    \draw[thin] (2,0) .. controls (2.7,-0.5) and (3.3,-0.5) .. (4,0);
    \draw[thin] (4,0) -- (6,0);
    \draw[thin] (0,0) .. controls (0.7,-1) and (5.3,-1) .. (6,0);
    \draw[fill] (2,0) circle(0.07);
    \draw[fill] (4,0) circle(0.07);
    \node at (1.2,0.3) {$f_1$};
    \node at (3,0.6) {$g$};
    \node at (3,-0.1) {$h$};
    \node at (4.8,0.3) {$f_2$};
    \node at (3,-1) {$e$};
    \draw[fill,white] (0,0) circle(0.5);
    \draw[thin] (0,0) circle(0.5);
    \node at (0,0) {$\Ga_1$};
    \draw[fill,white] (6,0) circle(0.5);
    \draw[thin] (6,0) circle(0.5);
    \node at (6,0) {$\Ga_2$};    

\end{tikzpicture}
\end{center}
The subgraphs $\Ga_1$ and $\Ga_2$ are unbalanced graphs of genus one, and the cycle formed by $g$ and $h$ is balanced. Denote by $\tGa'_1\to\Ga'_1$ and $\tGa'_2\to\Ga'_2$ the double covers over the graphs $\Ga'_1=\Ga_1\cup\{f_1,g,h\}$ and $\Ga'_2=\Ga_2\cup\{f_2,g,h\}$. There are four domains of polynomiality. If $e\leq f_1+f_2$, then
\[
q(\tGa/\Ga)=p_2(e,f_1+f_2)(g+h)+12e^2gh,
\]
where we note that $g+h=w_0(\tGa'/\Ga')$, where $\tGa'\to \Ga'$ is the double cover over the graph $\Ga'$ obtained from $\Ga$ by removing either $\Ga_1$ or $\Ga_2$. For $f_1+f_2\leq e$ we instead have 
\[
q(\tGa/\Ga)=p_3(e-f_1-f_2,g,h)+
p_2(f_1,e)w_0(\tGa'_2/\Ga'_2)+p_2(f_2,e)w_0(\tGa'_1/\Ga'_1)+
\]
\[
+12(2e-f_1-f_2)\left[(f_1+f_2)gh+f_1f_2(g+h)\right].
\]

\subsubsection*{Two $FS_2$-sets, two $FS_3$-sets} The $FS_2$-sets are disjoint, while the remaining five pairwise intersections are distinct edges, therefore we can label the sets as $\{e_1,f_1\}$, $\{e_2,f_2\}$, $\{e_1,f_2,g\}$, and $\{e_2,f_1,g\}$. The graph $\Ga$ has the following form:
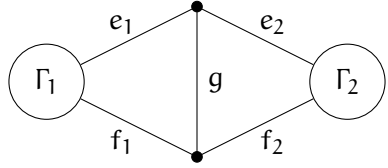
\begin{center}
\begin{tikzpicture}
    \draw[thin] (0,0) -- (2,1) -- (4,0);
    \draw[thin] (0,0) -- (2,-1) -- (4,0);
    \draw[thin] (2,1) -- (2,-1);
    \draw[fill,white] (0,0) circle(0.5);
    \draw[thin] (0,0) circle(0.5);
    \node at (0,0) {$\Ga_1$};
    \draw[fill,white] (4,0) circle(0.5);
    \draw[thin] (4,0) circle(0.5);
    \node at (4,0) {$\Ga_2$};
    \node[above] at (1,0.5) {$e_1$};
    \node[below] at (1,-0.5) {$f_1$};
    \node[above] at (3,0.5) {$e_2$};
    \node[below] at (3,-0.5) {$f_2$};
    \node[right] at (2,0) {$g$};
    \draw[fill] (2,1) circle(0.07);
    \draw[fill] (2,-1) circle(0.07);
\end{tikzpicture}
    \end{center}
The subgraphs $\Ga_1$ and $\Ga_2$ are unbalanced graphs of genus one. We denote by $\tGa'_1\to \Ga'_1$ and $\tGa'_2\to \Ga'_2$ the double covers over the subgraphs $\Ga'_1=\Ga_1\cup\{e_1,f_1,g\}$ and $\Ga'_2=\Ga_2\cup\{e_2,f_2,g\}$.

There are eight domains of polynomiality. On the cone $e_1\geq f_1$, $e_2\geq f_2$ we set
\[
q(\tGa/\Ga)=p_2(e_1,f_1)w_0(\tGa'_2/\Ga'_2)+p_2(e_2,f_2)w_0(\tGa'_1/\Ga'_1)+
\]
\[
+12g(e_1f_2^2+e_2f_1^2+2(e_1+e_2)f_1f_2-(f_1+f_2)f_1f_2)+24e_1e_2f_1f_2,
\]
while on the cone $f_1\geq e_1$, $f_2\geq e_2$ we exchange the $e$'s and the $f$'s. On the cone $e_1\geq f_1$, $f_2\geq e_2$ we set
\[
q(\tGa/\Ga)=p_2(e_1,f_1)w_0(\tGa'_2/\Ga'_2)+p_2(e_2,f_2)w_0(\tGa'_1/\Ga'_1)+p_3(e_1-f_1,f_2-e_2,g)+
\]
\[
+12g(-e_1e_2^2-f_1^2f_2+e_2^2f_1+e_2f_1^2+2e_1e_2f_2+2e_1f_1f_2)+24e_1e_2f_1f_2,
\]
and there are three domains of polynomiality, depending on which of the three quantities $e_1-f_1$, $f_2-e_2$, and $g$ is the smallest. Finally, on the cone $e_2\geq f_2$, $f_1\geq e_1$, the formula for $q(\tGa/\Ga)$ is obtained from the one above by exchanging the $e$'s and the $f$'s.

\subsubsection*{Three $FS_3$-sets} The union of the $FS_2$-sets consists of three edges, therefore we denote the three sets by $\{e_1,e_2\}$, $\{e_1,e_3\}$, and $\{e_2,e_3\}$. The graph $\Ga$ has the following form:
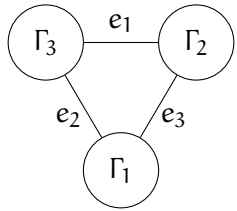
\begin{center}
\begin{tikzpicture}
    \draw[thin] (-1,0) -- (1,0) -- (0,-1.7) -- (-1,0);
    \node[above] at (0,0) {$e_1$};
    \node at (0.7,-1) {$e_3$};
    \node at (-0.7,-1) {$e_2$};
    \draw[fill,white] (-1,0) circle(0.5);
    \draw[thin] (-1,0) circle(0.5);
    \node at (-1,0) {$\Ga_3$};
    \draw[fill,white] (1,0) circle(0.5);
    \draw[thin] (1,0) circle(0.5);
    \node at (1,0) {$\Ga_2$};
    \draw[fill,white] (0,-1.7) circle(0.5);
    \draw[thin] (0,-1.7) circle(0.5);
    \node at (0,-1.7) {$\Ga_1$};
\end{tikzpicture}
\end{center}
The three subgraphs $\Ga_1$, $\Ga_2$, and $\Ga_3$ are unbalanced. We denote by $\Ga'_1$, $\Ga'_2$, and $\Ga'_3$ the genus two subgraphs obtained by removing respectively $\{e_2,e_3\}$, $\{e_1,e_3\}$, and $\{e_1,e_2\}$ from $\Ga$, and $\tGa'_i\to \Ga'_i$ the associated double covers. Let 
\[
s_1=e_1+e_2+e_3,\quad s_2=e_1e_2+e_1e_3+e_2e_3,\quad s_3=e_1e_2e_3
\]
denote the elementary symmetric functions in the $e_i$, and we introduce the quantity
\[
S(e_1,e_2,e_3)=2s_1^4-16s_1^2s_2+32s_2^2+16s_1s_3.
\]
There are two possible choices of graphs for each $\Ga_i$, giving four (up to isomorphism) possible graphs $\Ga$, having 0, 1, 2, or 3 $FS_3$-sets.

\subsubsection*{Three $FS_2$-sets, no $FS_3$-set} The graph $\Ga$ has the form
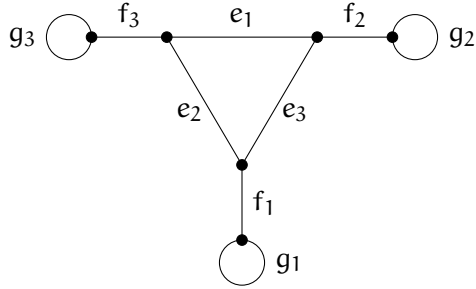
\begin{center}
\begin{tikzpicture}
    \draw[thin] (-1,0) -- (1,0) -- (0,-1.7) -- (-1,0);
    \node[above] at (0,0) {$e_1$};
    \node at (0.7,-1) {$e_3$};
    \node at (-0.7,-1) {$e_2$};
    \draw[fill] (-1,0) circle(0.07);
    \draw[fill] (1,0) circle(0.07);
    \draw[fill] (0,-1.7) circle(0.07);
    \draw[thin] (-1,0) -- (-2,0);
    \draw[thin] (1,0) -- (2,0);
    \draw[thin] (0,-1.7) -- (0,-2.7);
    \node[above] at (-1.5,0) {$f_3$};
    \node[above] at (1.5,0) {$f_2$};
    \node[right] at (0,-2.2) {$f_1$};
    \draw[fill] (-2,0) circle(0.07);
    \draw[fill] (2,0) circle(0.07);
    \draw[fill] (0,-2.7) circle(0.07);
    \draw[thin] (-2.3,0) circle(0.3);
    \draw[thin] (2.3,0) circle(0.3);
    \draw[thin] (0,-3) circle(0.3);
    \node[left] at (-2.6,0) {$g_3$};
    \node[right] at (2.6,0) {$g_2$};
    \node[right] at (0.3,-3) {$g_1$};
\end{tikzpicture}
\end{center}
The three loops $g_1,g_2,g_3$ are odd, while the remaining edges are even. There are four domains of polynomiality. On the cone $e_1\geq e_2+e_3$ we set 
\[
q(\tGa/\Ga)=p_2(e_2,e_1)w_0(\tGa'_3/\Ga'_3)+p_2(e_3,e_1)w_0(\tGa'_2/\Ga'_2)+
\]
\[
+12e_2e_3\left[(2e_1-e_2-e_3)(4f_1+g_1)+e_2(4f_2+g_2)+e_3(4f_3+g_3)\right].
\]
Cyclically permuting the indices, we obtain the formula for $q(\tGa/\Ga)$ on the cones $e_2\geq e_1+e_3$ and $e_3\geq e_1+e_2$. On the remaining central cone $e_1\leq e_2+e_3$, $e_2\leq e_1+e_3$, $e_3\leq e_1+e_2$ we set
\[
q(\tGa/\Ga)=S(e_1,e_2,e_3)+\sum_{i=1}^3p_2(e_i,s_1-e_i)(4f_i+g_i).
\]

\subsubsection*{Three $FS_2$-sets, one $FS_3$-set}  The graph $\Ga$ has the form
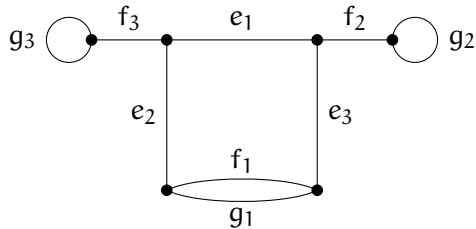
\begin{center}
    \begin{tikzpicture}
        \draw[thin] (0,0) -- (0,2) -- (2,2) -- (2,0);
        \draw[thin] (-1,2) -- (0,2);
        \draw[thin] (2,2) -- (3,2);
        \draw[thin] (-1.3,2) circle(0.3);
        \draw[thin] (3.3,2) circle(0.3);
        \draw[fill] (-1,2) circle(0.07);
        \draw[fill] (0,2) circle(0.07);
        \draw[fill] (2,2) circle(0.07);
        \draw[fill] (3,2) circle(0.07);
        \draw[fill] (0,0) circle(0.07);
        \draw[fill] (2,0) circle(0.07);
        \draw[thin] (0,0) .. controls (0.4,0.2) and (1.6,0.2) .. (2,0);
        \draw[thin] (0,0) .. controls (0.4,-0.2) and (1.6,-0.2) .. (2,0);
        \node[above] at (1,2) {$e_1$};
        \node[above] at (1,0.1) {$f_1$};
        \node[below] at (1,-0.1) {$g_1$};
        \node[left] at (0,1) {$e_2$};
        \node[right] at (2,1) {$e_3$};
        \node[above] at (-0.5,2) {$f_3$};
        \node[above] at (2.5,2) {$f_2$};
        \node[left] at (-1.6,2) {$g_3$};
        \node[right] at (3.6,2) {$g_2$};
    \end{tikzpicture}
\end{center}
The loops $g_2$ and $g_3$ and the cycle $f_1,g_1$ are odd, and the $FS_3$-set is $\{e_1,f_1,g_1\}$. There are six domains of polynomiality. On the cone $e_1\geq e_2+e_3$ we set
\[
q(\tGa/\Ga)=p_2(e_1,e_2)w_0(\tGa'_3/\Ga'_3)+p_2(e_3,e_1)w_0(\tGa'_2/\Ga'_2)+p_3(e_1-e_2-e_3,f_1,g_1)+
\]
\[
+12\left[
(2e_1-e_2-e_3)((e_2+e_3)f_1g_1+e_2e_3(f_1+g_1))+e_3^2e_2(4f_3+g_3)+e_3e_2^2(4f_2+g_2)
\right],
\]
and there are three domains of polynomiality, depending on which of the quantities $e_1-e_2-e_3$, $f_1$, and $g_1$ is the smallest. On the cone $e_2\geq e_1+e_3$ we set
\[
q(\tGa/\Ga)=p_2(e_1,e_2)w_0(\tGa'_2/\Ga'_2)+p_2(e_2,e_3)w_0(\tGa'_1/\Ga'_1)+
\]
\[
+12\left[e_1e_3(e_1(f_1+g_1)+(2e_2-e_1-e_3)(4f_2+g_2)-e_3(4f_3+g_3))+e_1^2f_1g_1\right].
\]
Exchanging 2 and 3, we obtain the formula for $q(\tGa/\Ga)$ on the cone $e_3\geq e_1+e_2$. Finally, on the central cone $e_1\leq e_2+e_3$, $e_2\leq e_1+e_3$, $e_3\leq e_1+e_2$ we set
\[
q(\tGa/\Ga)=S(e_1,e_2,e_3)+p_2(e_1,e_2+e_3)(f_1+g_1)+12e_1^2f_1g_1+\sum_{i=2}^3 p_2(e_i,s_1-e_i)(4f_i+g_i).
\]

\subsubsection*{Three $FS_2$-sets, two $FS_3$-sets} The graph $\Ga$ has the form
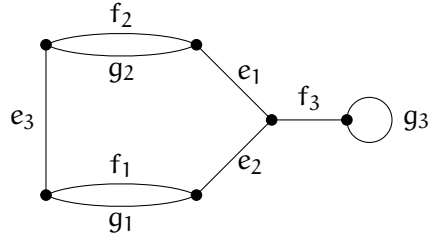
\begin{center}
    \begin{tikzpicture}
        \draw[thin] (0,0) -- (0,2);
        \draw[thin] (2,2) -- (3,1) -- (2,0);
        \draw[thin] (3,1) -- (4,1);
        \draw[thin] (0,0) .. controls (0.4,0.2) and (1.6,0.2) .. (2,0);
        \draw[thin] (0,0) .. controls (0.4,-0.2) and (1.6,-0.2) .. (2,0);
        \draw[thin] (0,2) .. controls (0.4,2.2) and (1.6,2.2) .. (2,2);
        \draw[thin] (0,2) .. controls (0.4,1.8) and (1.6,1.8) .. (2,2);
        \draw[fill] (0,0) circle(0.07);
        \draw[fill] (2,0) circle(0.07);
        \draw[fill] (0,2) circle(0.07);
        \draw[fill] (2,2) circle(0.07);
        \draw[fill] (3,1) circle(0.07);
        \draw[fill] (4,1) circle(0.07);
        \draw[thin] (4.3,1) circle(0.3);
        \node[right] at (2.4,1.6) {$e_1$};
        \node[above] at (1,0.1) {$f_1$};
        \node[below] at (1,-0.1) {$g_1$};
        \node[right] at (2.4,0.4) {$e_2$};
        \node[above] at (1,2.1) {$f_2$};
        \node[below] at (1,1.9) {$g_2$};
        \node[left] at (0,1) {$e_3$};
        \node[above] at (3.5,1) {$f_3$};
        \node[right] at (4.6,1) {$g_3$};
    
    \end{tikzpicture}
\end{center}
The loop $g_3$ and the cycles $f_1,g_1$ and $f_2,g_2$ are odd, and the $FS_3$-sets are $\{e_1,f_1,g_1\}$ and $\{e_2,f_2,g_2\}$. There are eight domains of polynomiality. On the cone $e_1\geq e_2+e_3$ we set
\[
q(\tGa/\Ga)=p_2(e_1,e_2)w_0(\tGa'_3/\Ga'_3)+p_2(e_3,e_1)w_0(\tGa'_2/\Ga'_2)+p_3(e_1-e_2-e_3,f_1,g_1)+
\]
\[
+12\left[
(2e_1-e_2-e_3)((e_2+e_3)f_1g_1+e_2e_3(f_1+g_1))+e_3^2e_2(4f_3+g_3)+e_3e_2^2(f_2+g_2)+e_2^2f_2g_2
\right],
\]
with three domains of polynomiality determined by which of the quantities $e_1-e_2-e_3$, $f_1$, and $g_1$ is the smallest. On the cone $e_2\geq e_1+e_3$ we exchange 1 and 2. On the cone $e_3\geq e_1+e_2$ we set
\[
q(\tGa/\Ga)=p_2(e_1,e_3)w_0(\tGa'_3/\Ga'_3)+p_2(e_2,e_3)w_0(\tGa'_1/\Ga'_1)+
\]
\[
+12\left[e_1e_2(e_1(f_1+g_1)+e_2(f_2+g_2)+(2e_3-e_1-e_2)(4f_3+g_3))+e_1^2f_1g_1+e_2^2f_2g_2
\right].
\]
Finally, on the central cone $e_1\leq e_2+e_3$, $e_2\leq e_1+e_3$, $e_3\leq e_1+e_2$ we set
\[
q(\tGa/\Ga)=S(e_1,e_2,e_3)+\sum_{i=1}^2[p_2(e_i,s-e_i)(f_i+g_i)+12e_i^2f_ig_i]+p_2(e_3,e_1+e_2)(4f_3+g_3).
\]

\subsubsection*{Three $FS_2$-sets, three $FS_3$-sets} The graph $\Ga$ has the form
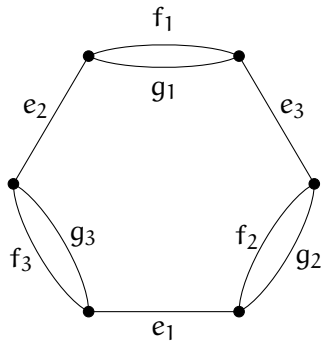
\begin{center}
    \begin{tikzpicture}
        \draw[thin] (0,0) -- (2,0);
        \draw[thin] (3,1.7) -- (2,3.4);
        \draw[thin] (0,3.4) -- (-1,1.7);
        \draw[thin] (0,3.4) .. controls (0.4,3.6) and (1.6,3.6) .. (2,3.4);
        \draw[thin] (0,3.4) .. controls (0.4,3.2) and (1.6,3.2) .. (2,3.4);
        \draw[thin] (2,0) .. controls (2,0.5) and (2.6,1.5) .. (3,1.7);
        \draw[thin] (2,0) .. controls (2.4,0.2) and (3,1.2) .. (3,1.7);
        \draw[thin] (0,0) .. controls (0,0.5) and (-0.6,1.5) .. (-1,1.7);
        \draw[thin] (0,0) .. controls (-0.4,0.2) and (-1,1.2) .. (-1,1.7);
        \draw[fill] (0,0) circle(0.07);
        \draw[fill] (2,0) circle(0.07);
        \draw[fill] (0,3.4) circle(0.07);
        \draw[fill] (2,3.4) circle(0.07);
        \draw[fill] (3,1.7) circle(0.07);
        \draw[fill] (-1,1.7) circle(0.07);
        \node[below] at (1,0) {$e_1$};
        \node[above] at (1,3.6) {$f_1$};
        \node[below] at (1,3.2) {$g_1$};
        \node[left] at (-0.4, 2.7) {$e_2$};
        \node[left] at (2.4, 1) {$f_2$};
        \node[right] at (2.6, 0.7) {$g_2$};
        \node[right] at (2.4, 2.7) {$e_3$};
        \node[left] at (-0.6, 0.7) {$f_3$};
        \node[right] at (-0.4, 1) {$g_3$};
    \end{tikzpicture}
\end{center}
The cycles $f_1,g_1$, $f_2,g_2$, and $f_3,g_3$ are odd, and the $FS_3$-sets are $\{e_1,f_1,g_1\}$, $\{e_2,f_2,g_2\}$, and $\{e_3,f_3,g_3\}$. There are ten domains of polynomiality. On the cone $e_3\geq e_1+e_2$ we set
\[
q(\tGa/\Ga)=p_2(e_1,e_3)w_0(\tGa'_2/\Ga'_2)+p_2(e_2,e_3)w_0(\tGa'_1/\Ga'_1)+p_3(e_3-e_1-e_2,f_3,g_3)+
\]
\[
+12\left[(f_3g_3(e_1+e_2)+(f_3+g_3)e_1e_2)(2e_3-e_1-e_2)
+e_1^2f_1g_1+e_2^2g_2f_2+e_1^2e_2(f_1+g_1)+e_1e_2^2(f_2+g_2)\right].
\]
There are three domains of polynomiality on this cone, corresponding to which of the quantities $e_3-e_1-e_2$, $f_3$, and $g_3$ is the smallest. The formula for $q(\tGa/\Ga)$ on the cones $e_1\geq e_2+e_3$ and $e_2\geq e_1+e_3$ is obtained by cyclically permuting the indices. Finally, on the central cone $e_1\leq e_2+e_3$, $e_2\leq e_1+e_3$, $e_3\leq e_1+e_2$ we set 
\[
q(\tGa/\Ga)=S(e_1,e_2,e_3)+\sum_{i=1}^3\left[p_2(e_i,s_1-e_i)(f_i+g_i)+12e_i^2f_ig_i\right].
\]

\bibliographystyle{amsalpha}
\bibliography{references}{}

\providecommand{\bysame}{\leavevmode\hbox to3em{\hrulefill}\thinspace}
\providecommand{\MR}{\relax\ifhmode\unskip\space\fi MR }
\providecommand{\MRhref}[2]{%
  \href{http://www.ams.org/mathscinet-getitem?mr=#1}{#2}
}
\providecommand{\href}[2]{#2}
\begin{thebibliography}{CMGHL17}

\bibitem[ABH02]{2002AlexeevBirkenhakeHulek}
Valery Alexeev, Ch~Birkenhake, and Klaus Hulek, \emph{Degenerations of prym
  varieties}, J. reine angew. Math. \textbf{553} (2002), 73--116.

\bibitem[ABKS14]{2014AnBakerKuperbergShokrieh}
Yang An, Matthew Baker, Greg Kuperberg, and Farbod Shokrieh, \emph{Canonical
  representatives for divisor classes on tropical curves and the matrix–tree
  theorem}, Forum Math., Sigma \textbf{2} (2014), e24, 25 pages.

\bibitem[BF11]{baker2011metric}
Matthew Baker and Xander Faber, \emph{Metric properties of the tropical
  {A}bel--{J}acobi map}, J. Algebraic Combin. \textbf{33} (2011), no.~3,
  349--381. \MR{2772537}

\bibitem[BMV11]{2011BrannettiMeloViviani}
Silvia Brannetti, Margarida Melo, and Filippo Viviani, \emph{On the tropical
  {T}orelli map}, Advances in Mathematics \textbf{226} (2011), no.~3,
  2546--2586.

\bibitem[BN07]{baker2007riemann}
Matthew Baker and Serguei Norine, \emph{Riemann--{R}och and {A}bel--{J}acobi
  theory on a finite graph}, Adv. Math. \textbf{215} (2007), no.~2, 766--788.
  \MR{2355607}

\bibitem[CD18]{cools2018metric}
Filip Cools and Jan Draisma, \emph{On metric graphs with prescribed gonality},
  J. Combin. Theory Ser. A \textbf{156} (2018), 1--21. \MR{3762100}

\bibitem[CMGHL17]{2017Casalaina-MartinGrushevskyHulekLaza}
Sebastian Casalaina-Martin, Samuel Grushevsky, Klaus Hulek, and Radu Laza,
  \emph{Extending the {P}rym map to toroidal compactifications of the moduli
  space of abelian varieties (with an appendix by {M}athieu {D}utour
  {S}ikiri{\'c})}, Journal of the {E}{M}{S} \textbf{19} (2017), no.~3,
  659--723.

\bibitem[CS13]{2013ConwaySloane}
John~Horton Conway and Neil James~Alexander Sloane, \emph{Sphere packings,
  lattices and groups}, vol. 290, Springer Science \& Business Media, 2013.

\bibitem[CV10a]{CaporasoViviani}
Lucia Caporaso and Filippo Viviani, \emph{Torelli theorem for graphs and
  tropical curves}, Duke Math. J. \textbf{153} (2010), no.~1, 129--171.
  \MR{2641941}

\bibitem[CV10b]{2010CaporasoViviani}
\bysame, \emph{Torelli theorem for graphs and tropical curves}, Duke Math. J.
  \textbf{153} (2010), no.~1, 129--171. \MR{2641941}

\bibitem[dJS22]{2022deJongShokrieh}
Robin de~Jong and Farbod Shokrieh, \emph{Faltings height and
  {N}{\'e}ron--{T}ate height of a theta divisor}, Compositio Mathematica
  \textbf{158} (2022), no.~1, 1--32.

\bibitem[dJS23]{2023deJongShokrieh}
\bysame, \emph{Tropical moments of tropical {J}acobians}, Canadian Journal of
  Mathematics \textbf{75} (2023), no.~4, 1045--1075.

\bibitem[FS86]{1986FriedmanSmith}
Robert Friedman and Roy Smith, \emph{Degenerations of {P}rym varieties and
  intersections of three quadrics}, Inventiones mathematicae \textbf{85}
  (1986), no.~3, 615--635.

\bibitem[Ger82]{1982Gerritzen}
L.~Gerritzen, \emph{Die {J}acobi-{A}bbildung \"{u}ber dem {R}aum der
  {M}umfordkurven}, Math. Ann. \textbf{261} (1982), no.~1, 81--100. \MR{675209}

\bibitem[GZ24]{2024GhoshZakharov}
Arkabrata Ghosh and Dmitry Zakharov, \emph{The {P}rym variety of a dilated
  double cover of metric graphs}, Annals of Combinatorics (2024), 1--22.

\bibitem[JL18]{2018JensenLen}
David Jensen and Yoav Len, \emph{Tropicalization of theta characteristics,
  double covers, and {P}rym varieties}, Selecta Math. (N.S.) \textbf{24}
  (2018), no.~2, 1391--1410. \MR{3782424}

\bibitem[LU21]{LenUlirsch}
Yoav Len and Martin Ulirsch, \emph{Skeletons of {P}rym varieties and
  {B}rill--{N}oether theory}, Algebra Number Theory \textbf{15} (2021), no.~3,
  785--820. \MR{4261102}

\bibitem[LZ22]{2022LenZakharov}
Yoav Len and Dmitry Zakharov, \emph{Kirchhoff's theorem for {P}rym varieties},
  Forum Math. Sigma \textbf{10} (2022), Paper No. e11, 54. \MR{4382460}

\bibitem[MZ08]{MikhalkinZharkov}
Grigory Mikhalkin and Ilia Zharkov, \emph{Tropical curves, their {J}acobians
  and theta functions}, Curves and abelian varieties, Contemp. Math., vol. 465,
  Amer. Math. Soc., Providence, RI, 2008, pp.~203--230. \MR{2457739}

\bibitem[Nam80]{1980Namikawa}
Yukihiko Namikawa, \emph{Toroidal compactification of siegel spaces}, vol. 812,
  Springer, 1980.

\bibitem[Rec74]{1974Recillas}
Sevin Recillas, \emph{Jacobians of curves with {$g^{1}_{4}$}'s are the {P}rym's
  of trigonal curves}, Bol. Soc. Mat. Mexicana (2) \textbf{19} (1974), no.~1,
  9--13. \MR{480505}

\bibitem[RSW23]{2023RichmanShokriehWu}
Harry Richman, Farbod Shokrieh, and Chenxi Wu, \emph{Counting two-forests and
  random cut size via potential theory}, arXiv preprint arXiv:2308.03859
  (2023).

\bibitem[RZ23]{2023RoehrleZakharov}
Felix R{\"o}hrle and Dmitry Zakharov, \emph{A matroidal perspective on the
  tropical {P}rym variety}, arXiv preprint arXiv:2311.09872 (2023).

\bibitem[RZ25]{2022RoehrleZakharov}
\bysame, \emph{The tropical $ n $-gonal construction}, Algebraic Combinatorics
  \textbf{8} (2025), no.~2, 319--378.

\bibitem[Vol02]{2002Vologodsky}
Vitaly Vologodsky, \emph{The locus of indeterminacy of the prym map}, J. reine
  angew. Math. \textbf{553} (2002), 117--124.

\bibitem[Zas82]{1982Zaslavsky}
Thomas Zaslavsky, \emph{Signed graphs}, Discrete Applied Mathematics \textbf{4}
  (1982), no.~1, 47--74.

\end{thebibliography}

\end{document}